\documentclass{my_aims}

\usepackage{amsmath}
\usepackage{paralist}
\usepackage{graphics}
\usepackage{epsfig}
\usepackage{graphicx}
\usepackage{epstopdf}
\usepackage[colorlinks=true]{hyperref}

\usepackage{subfig}

\hypersetup{urlcolor=blue, citecolor=red}
\usepackage{hyperref}

\newtheorem{theorem}{Theorem}[section]

\newtheorem{lemma}[theorem]{Lemma}

\theoremstyle{definition}

\newtheorem{remark}{Remark}

\title[Optimal control of a TB model with delays]{Optimal control
of a Tuberculosis model\\ with state and control delays}

\author[C. J. Silva, H. Maurer and D. F. M. Torres]{}

\subjclass{Primary: 34D30; 92D30 Secondary: 49M05; 93A30}

\keywords{Tuberculosis; time delays; stability; optimal control.}

 \email{cjoaosilva@ua.pt}
 \email{maurer@math.uni-muenster.de}
 \email{delfim@ua.pt}

\thanks{The first author is supported
by the FCT post-doc grant SFRH/BPD/72061/2010.}

\thanks{$^*$Corresponding author: maurer@math.uni-muenster.de}

\begin{document}

\maketitle

\centerline{\scshape Cristiana J. Silva}
\medskip
{\footnotesize
\centerline{Center for Research and Development in Mathematics and Applications (CIDMA)}
\centerline{Department of Mathematics, University of Aveiro, 3810--193 Aveiro, Portugal}
}

\medskip

\centerline{\scshape Helmut Maurer$^*$}
\medskip
{\footnotesize
\centerline{Institute of Computational and Applied Mathematics}
\centerline{University of M\"{u}nster, D-48149 M\"{u}nster, Germany}
}

\medskip

\centerline{\scshape Delfim F. M. Torres}
\medskip
{\footnotesize
\centerline{Center for Research and Development in Mathematics and Applications (CIDMA)}
\centerline{Department of Mathematics, University of Aveiro, 3810--193 Aveiro, Portugal}
}

\bigskip


\begin{abstract}
We introduce delays in a tuberculosis (TB) model, representing
the time delay on the diagnosis and commencement
of treatment of individuals with active TB infection.
The stability of the disease free and endemic equilibriums is investigated
for any time delay. Corresponding optimal control problems,
with time delays in both state and control variables,
are formulated and studied. Although it is well-known that
there is a delay between two to eight weeks between TB infection
and reaction of body's immune system to tuberculin,
delays for the active infected to be detected and treated,
and delays on the treatment of persistent latent individuals
due to clinical and patient reasons, which clearly
justifies the introduction of time delays on state and control measures,
our work seems to be the first to consider such time-delays for TB and
apply time-delay optimal control to carry out the optimality analysis.
\end{abstract}


\section{Introduction}

Tuberculosis (TB) is the second leading cause of death
from an infectious disease worldwide \cite{WHO:2014}.
Active TB refers to disease that occurs in someone infected with
\emph{Mycobacterium tuberculosis}. It is characterized by signs or symptoms of
active disease, or both, and is distinct from latent tuberculosis infection,
which occurs without signs or symptoms of active disease. Only individuals
with active TB can transmit the infection. Many people
with active TB do not experience typical TB symptoms in the early stages of
the disease. These individuals are unlikely to seek care early, and may not be
properly diagnosed when seeking care \cite{WHO:screen:TB:2013}.

Delays to diagnosis of active TB present a major obstacle to the control
of a TB epidemic \cite{UysEtAall:PlosOne:2007}, it may worsen the disease,
increase the risk of death and enhance tuberculosis transmission
to the community \cite{delay:diag:review:BMC:2009,Toman:1979}.
Both patient and the health system may be responsible
for the diagnosis delay \cite{delay:diag:review:BMC:2009}.
Efforts should be done in patient knowledge/awareness about TB,
and health care systems should improve case finding strategies
to reduce the delay in diagnosis of active TB
\cite{Lambert:TMIH:2005,delay:diag:review:BMC:2009,Storla:BMC:2008}.

Mathematical models are an important tool in analyzing the spread and control
of infectious diseases \cite{Hethcote:2000}. There are several mathematical
dynamic models for TB, see, e.g.,
\cite{Chavez:JMB:1997,Cohen:NM:2004,Gomes_etall_2007,Rodrigues:TPB:2007}.
In this paper we consider the mathematical model for TB proposed
in \cite{Gomes_etall_2007}. We introduce a discrete time delay
which represents the delay on the diagnosis of individuals
with active TB and commencement of treatment. The stability
of the disease free and endemic equilibriums is analyzed for any time delay.

Optimal control theory has been successfully applied to TB mathematical models
(see, e.g., \cite{RodSilvaTorres2014,MR2970904,SilvaTorres:review} and references
cited therein). We propose and analyze an optimal control problem where the
control system is the mathematical model from \cite{Gomes_etall_2007}, but with
a time delay in the state variable that represents individuals with active TB,
and introduce two control functions. The control functions represent the fraction
of early and persistent latent individuals that are treated for TB. Treatment
of latent TB infection greatly reduces the risk that TB infection will progress
to active TB disease. Certain groups are at very high risk of developing active
TB disease once infected. Every effort should be made to begin appropriate
treatment and to ensure completion of the entire course of treatment for
latent TB infection \cite{url:center:diseases}. Treatment of latent TB
infection should be initiated after the possibility of TB disease has been excluded.
It can take 2 to 8 weeks after TB infection for the body's immune system to react
to tuberculin and for the infected to be detected, which justifies the introduction
of a time delay on the control associated to treatment of early latent individuals.
On the other hand, delays in the treatment of latent TB may also occur 
due to clinical and demographic patient and health care services characteristics.
For these reasons, we consider discrete time delays in both control functions.
To our knowledge, this work is the first to apply optimal control
theory to a TB model with time delay in state and control variables.

The paper is organized as follows. In Section~\ref{sec:2}
we formulate the TB model with state delay. The stability
of the disease free equilibrium is analyzed in Section~\ref{sec:3}
while stability of the endemic equilibrium is investigated in
Section~\ref{sec:4}. Optimal control of TB with state and control delays
is carried out in Section~\ref{sec:5} and some numerical results given in
Section~\ref{sec:6}. We end with Section~\ref{sec:7} of conclusions.


\section{TB model with state delay}
\label{sec:2}

In this section we consider a TB mathematical model proposed in \cite{Gomes_etall_2007},
where reinfection and post-exposure interventions for tuberculosis are considered.
The model divides the total population into five categories: susceptible ($S$);
early latent ($L_1$), \textrm{i.e.}, individuals recently infected (less than
two years) but not infectious; infected ($I$), \textrm{i.e.}, individuals who
have active TB and are infectious; persistent latent ($L_2$), \textrm{i.e.},
individuals who were infected and remain latent; and recovered ($R$),
\textrm{i.e.}, individuals who were previously infected and have been treated.
As in \cite{Gomes_etall_2007}, we assume that the total population, $N$,
with $N = S(t) + L_1(t) + I(t) + L_2(t) + R(t)$, is constant in time. In other
words, we assume that the birth and death rates are equal and there are no
TB-related deaths. We introduce a discrete time-delay in the state variable $I$,
denoted by $d_{I}$, that represents the delay to diagnosis and commencement
of treatment of active TB infection,
\begin{equation}
\label{TBmodel:delay:state}
\begin{cases}
\dot{S}(t) = \mu N - \frac{\beta}{N} I(t) S(t) - \mu S(t),\\
\dot{L_1}(t) = \frac{\beta}{N} I(t)\left( S(t)
+ \sigma L_2(t) + \sigma_R R(t)\right) - \left(\delta + \tau_1 + \mu\right)L_1(t),\\
\dot{I}(t) = \phi \delta L_1(t) + \omega L_2(t) + \omega_R R(t)
- \tau_0 I(t - d_I) - \mu I(t),\\
\dot{L_2}(t) = (1 - \phi) \delta L_1(t) - \sigma \frac{\beta}{N} I(t) L_2(t)
- (\omega  + \tau_2 + \mu)L_2(t),\\
\dot{R}(t) = \tau_0 I(t - d_I) +  \tau_1 L_1(t)
+ \tau_2 L_2(t) - \sigma_R \frac{\beta}{N} I(t) R(t) - (\omega_R + \mu)R(t) \, .
\end{cases}
\end{equation}
The initial conditions for system (\ref{TBmodel:delay:state}) are
\begin{equation}
\label{eq:init:cond:moddelay}
S(\theta) = \varphi_1(\theta), \, L_1(\theta) = \varphi_2(\theta), \,
I(\theta) = \varphi_3(\theta), \, L_2(\theta) = \varphi_4(\theta), \,
R(\theta) = \varphi_5(\theta),
\end{equation}
$-d_I \leq \theta \leq 0$, where
$\varphi=\left(\varphi_1, \varphi_2, \varphi_3, \varphi_4,
\varphi_5 \right)^T \in C$ with $C$ the Banach space
$C \left([-d_I, 0], {\mathbb{R}}^5 \right)$ of continuous functions
mapping the interval $[-d_I, 0]$ into ${\mathbb{R}}^5$.
From biological meaning, we further assume that
$\varphi_i(0) > 0$ for $i = 1, \ldots, 5$.

Throughout this paper, we focus on the dynamics of the solutions of
(\ref{TBmodel:delay:state}) in the restricted region
$$
\Omega = \left\{ (S, L_1, I, L_2, R) \in {\mathbb{R}}^5_{+0} \, | \,  0
\leq S + L_1 + I + L_2 + R = N \right\}.
$$
In this region, the usual local existence, uniqueness and continuation results
apply \cite{Hale_Lunel_book1993,YKuang_1993}. Hence, a unique solution
$\left( S(t), L_1(t), I(t), L_2(t), R(t)\right)$ of (\ref{TBmodel:delay:state})
with initial condition (\ref{eq:init:cond:moddelay}) exists for all time $t \geq 0$.
Consider the solutions of (\ref{TBmodel:delay:state}) with $(\varphi_1(\theta),
\ldots, \varphi_5(\theta)) \in Int \, \Omega$, $i = 1, \ldots, 5$,
for all $\theta \in [-d_I, 0]$. Then the solutions stay in the interior
of the region $\Omega$ for all time $t \geq 0$, i.e., the region $\Omega$
is positively invariant with respect to system (\ref{TBmodel:delay:state})
(see, e.g., \cite{YKuang_1993}).

A mathematical model has a disease free equilibrium
if it has an equilibrium point at which the population
remains in the absence of the disease \cite{van:den:Driessche:2002}.
The model (\ref{TBmodel:delay:state}) has a disease free equilibrium
given by $E_0 = (N, 0, 0, 0, 0)$.

The basic reproduction number $R_0$ represents the expected average
number of new TB infections produced by a single TB active infected
individual when in contact with a completely susceptible
population \cite{van:den:Driessche:2002}.
For model (\ref{TBmodel:delay:state}) it is given by
\begin{equation}
\label{eq:R0}
R_0 = \frac{\beta}{\mu} \frac{ \omega_R \left( \omega+ \tau_2+\mu
 \right) \tau_1 + \delta  \left[  \left( \omega_R +\mu \right)
\left( \phi \mu + \omega \right) + \left( \omega_R + \phi \mu
\right) \tau_2  \right]  }{\left( \tau_0 +\mu+\omega_R \right)
\left( \omega+\tau_2 +\mu \right) \left(\delta+\tau_1+\mu \right) }
= \frac{\mathcal{N}}{\mathcal{D}}\, .
\end{equation}
Note that in \cite{Gomes_etall_2007} the basic reproduction number is deduced
under the assumption that $\tau_1=\tau_2=0$. The expression (\ref{eq:R0})
generalizes the one given in \cite{Gomes_etall_2007}.


\section{Stability of the disease free equilibrium}
\label{sec:3}

It is important to analyze the stability of the disease free equilibrium,
as it indicates whether the population will remain in the absence of the disease,
or the disease will persist for all time
\cite{van:den:Driessche:2002,YangWei:DCDSB:2015}.
System \eqref{TBmodel:delay:state} is equivalent to
\begin{equation}
\label{TBmodel:delay:state:4eq}
\begin{cases}
\dot{S}(t) = \mu N - \frac{\beta}{N} I(t) S(t) - \mu S(t),\\
\dot{L_1}(t) = \frac{\beta}{N} I(t)\left( S(t)
+ \sigma L_2(t) + \sigma_R (N - S(t) - L_1(t) - I(t) - L_2(t))\right)\\
\quad \quad - \left(\delta + \tau_1 + \mu\right)L_1(t),\\
\dot{I}(t) = \phi \delta L_1(t) + \omega L_2(t)
+ \omega_R (N - S(t) - L_1(t) - I(t) - L_2(t))\\
\quad \quad - \tau_0 I(t - d_I) - \mu I(t),\\
\dot{L_2}(t) = (1 - \phi) \delta L_1(t) - \sigma \frac{\beta}{N} I(t) L_2(t)
- (\omega  + \tau_2 + \mu)L_2(t) \, ,
\end{cases}
\end{equation}
where the equation for $R(t)$ is derived from
$R(t) = N - S(t) - L_1(t) - I(t) - L_2(t)$.
The disease free equilibrium of system \eqref{TBmodel:delay:state:4eq}
is given by $(S_0, L_{10}, I_0, L_{20}) = (N, 0, 0, 0)$.
To discuss its local asymptotic stability,
let us consider the coordinate transformation
$s(t) = S(t) - \overline{S}$, $l_1(t) = L_1(t) - \overline{L}_1$,
$i(t) = I(t) - \overline{I}$, $l_2(t) = L_2(t) - \overline{L}_2$,
where $(\overline{S}, \overline{L}_1, \overline{I}, \overline{L}_2)$ denotes
any equilibrium of (\ref{TBmodel:delay:state:4eq}). Hence, we have that the
corresponding linearized system of (\ref{TBmodel:delay:state:4eq}) is of the form
\begin{equation}
\label{model:linear:delay:4eq}
\begin{cases}
\dot{s}(t) = -{\frac {\beta\,\overline{I}+\mu\,N}{N}} \, s(t)
-\frac {\beta}{N}\,\overline{S} \, i(t)\\[0.2 cm]
\dot{l}_1(t) = - \frac{\beta}{N}\,\overline{I}\,\left(\sigma_R -1 \right)\,s(t)
- \frac{\beta}{N} \left( \sigma_R \, \overline{I}
+ (\delta + \tau_1 + \mu) N \right) \, l_1(t)\\
\quad \quad \quad -\frac{\beta}{N} \left( -\overline{S}-\sigma\,\overline{L}_2
- \sigma_R (N+ \overline{S}+ \overline{L}_1
+2\,\overline{I}+ \overline{L}_2) \right) \, i(t)
+ {\frac {\beta\,\overline{I}\,
\left( \sigma- \sigma_R  \right) }{N}} \, l_2(t) \\[0.2 cm]
\dot{i}(t) =  - \omega_R \, s(t)  + \left(\phi\,\delta- \omega_R \right) l_1(t)
- \left(\omega_R -\mu \right)  i(t) +\left(\omega- \omega_R \right) l_2(t)
- \tau_0 \,i(t-d_I)\\[0.2 cm]
\dot{l}_2(t) =   \left( 1 - \phi \right) \delta \, l_1(t)
-\frac{\beta}{N}\sigma \,\overline{L}_2 \, i(t)
- \frac{\beta}{N} \left(\sigma \overline{I}
+ (\omega + \tau_2 + \mu) N \right)\, l_2(t) \, .
\end{cases}
\end{equation}
We then express system (\ref{model:linear:delay:4eq}) in matrix form as follows:
\begin{equation*}
\frac{d}{dt}\left( \begin{array}{c}
s(t)\\
l_1(t)\\
i(t)\\
l_2(t)
\end{array} \right) = A_1
\left( \begin{array}{c}
s(t)\\
l_1(t)\\
i(t)\\
l_2(t)
\end{array} \right) + A_2
\left( \begin{array}{c}
s(t-d_I)\\
l_1(t-d_I)\\
i(t-d_I))\\
l_2(t-d_I)
\end{array} \right)
\end{equation*}
with $A_1$ the $4 \times 4$ matrix
\begin{equation*}
A_1 = \left[ \begin {array}{cccc}
-{\frac {\beta\,\overline{I}+\mu\,N}{N}}
&0&-{\frac {\beta\,\overline{S}}{N}}&0\\ \noalign{\medskip}
{\frac {\beta\,\overline{I}\, \left( 1-\sigma_R \right)}{N}}
&-{\frac {\beta\,\overline{I}\,\sigma_R + c_1 N}{N}}
&{\frac {\beta\, \left( \overline{S}+\sigma\,\overline{L}_2
+\sigma_R (N+ \overline{S}+ \overline{L}_1
+2\,\overline{I}+\overline{L}_2)\right)}{N}}
&{\frac {\beta\,\overline{I}\,
\left( \sigma- \sigma_R  \right) }{N}}\\ \noalign{\medskip}
- \omega_R &\phi\,\delta- \omega_R &- \omega_R -\mu&\omega- \omega_R\\
\noalign{\medskip}
0&- \left( -1+\phi \right) \delta&-{\frac {\sigma \,\beta\,\overline{L}_2}{N}}
&-{\frac {\beta\,\overline{I}\,\sigma+ c_2 N}{N}}\end {array} \right],
\end{equation*}
where $c_1 = \delta + \tau_1 + \mu$ and $c_2 = \omega + \tau_2 + \mu$,
and $A_2$ the $4 \times 4$ diagonal matrix
$A_2 = \textrm{diag}(0, 0, -\tau_0, 0)$.
The transcendental characteristic equation of system
(\ref{model:linear:delay:4eq}) is defined by $\Delta(\lambda)
= \det \left( \lambda I - A_1 - {\rm e}^{-\lambda d_I} A_2 \right) = 0$
and is given by
\begin{equation}
\label{eq:caract}
\Delta(\lambda) =  P(\lambda) + Q(\lambda) =0,
\end{equation}
where
\begin{equation*}
\begin{split}
P(\lambda) &=  \lambda^4 + a_3 \lambda^3 + a_2 \lambda^2 + a_1 \lambda + a_0,\\
Q(\lambda) &= \tau_0  \left( \lambda+\mu \right) \left( \lambda+ c_1 \right)
\left( \lambda+ c_2 \right) \left( {\rm e}^{-\lambda d_I} -1 \right),
\end{split}
\end{equation*}
with $a_0 = \mathcal{D} - \mathcal{N}$,
\begin{equation*}
\begin{split}
a_1 &= \frac{2}{\mu} \, \mathcal{D} + \mu^2(c_1 + c_2 + c_4) - c_4 c_5 c_6
- \beta (\tau_1 \omega_R + \omega \delta +\delta \phi(\omega_R + \tau_2 + 2\mu)),\\
a_2 &= c_4 c_5 +3\, \mu (c_1 + c_2 + c_4)+ c_6 (c_4 + c_5) -\beta\,\phi\,\delta \, , \\
a_3 &= c_1 + c_2 + c_3 + \mu,
\end{split}
\end{equation*}
and $c_3 = \omega_R + \tau_0 +  \mu$, $c_4 = \tau_0 + \omega_R$,
$c_5 = \tau_2 + \omega$, $c_6 = \delta + \tau_1$.

\begin{remark}
For any $d_I \geq 0$ and $a_3 > 0$, if $R_0 < 1$, then $a_0 > 0$.
\end{remark}

Recall that an equilibrium point is asymptotically stable if all roots
of the corresponding characteristic equation have negative
real parts \cite{Bellman1963}.

\begin{lemma}
\label{lemma:E0:instable}
If $R_0 > 1$, then the disease free equilibrium $E_0$
is unstable for any $d_I \geq 0$.
\end{lemma}

\begin{proof}
The characteristic equation (\ref{eq:caract}) satisfies
$\Delta(0) = \mathcal{D} - \mathcal{N}$.
Assuming $R_0 > 1$, then $\Delta(0) < 0$. Since
$\lim_{\lambda \to +\infty} \Delta(\lambda) = + \infty$,
there exists at least one positive root of (\ref{eq:caract}).
\end{proof}

\begin{lemma}
If
(i) $R_0 < 1$,
(ii) $a_1 > 0$,
(iii) $a_2 > 0$,
(iv) $a_3a_2 > a_1$, and
(v)~$a_3a_2a_1 > a_1^2 + a_3^2 a_0$,
then the disease free equilibrium $E_0$
is locally asymptotically stable for $d_I = 0$.
\end{lemma}

\begin{proof}
When $d_I =0$, the associated transcendental characteristic equation
(\ref{eq:caract}) of system (\ref{model:linear:delay:4eq}) at
$(N, 0, 0, 0) = (\overline{S}, \overline{L}_1, \overline{I}, \overline{L}_2)$
becomes $\Delta(\lambda) = P(\lambda) = 0$.
Using the Routh--Hurwitz criterion for fourth-order polynomials, all the roots
of $P(\lambda)$ have negative real part if all coefficients satisfy
$a_n > 0$ for $n = 0, \ldots, 3$, $a_3a_2 > a_1$,
and $a_3a_2a_1 > a_1^2 + a_3^2 a_0$.
\end{proof}

In the case $d_I > 0$, by Rouch\'{e}'s theorem \cite{Dieudonne:1960},
if instability occurs for a particular value of the delay $d_I$,
then a characteristic root of (\ref{eq:caract}) must intersect the imaginary axis.
Our aim is to prove that the polynomial (\ref{eq:caract}) does not have purely
imaginary roots for all positive delays (see, e.g.,
\cite{BuonomoMBE2015,CulshawMBS2000}). The complex $i\, b$, $b > 0$,
is a root of (\ref{eq:caract}) if and only if ${b}^{4} - i a_3 {b}^{3}
- a_2 {b}^{2} + i a_1 b+ a_0 + \tau_0 \left( i b +\mu \right)
\left( i b+ c_1 \right)  \left(i b+ c_2 \right)
\left( {{\rm e}^{-i b \, d_I }}-1 \right) = 0$.
By using Euler's formula $\exp^{-i b d_I} = \cos(b d_I) -i \sin(b d_I)$,
and by separating real and imaginary parts, we have
\begin{multline*}
a_3 {b}^{3} + \tau_0 b c_1 c_2 - a_1 b -\tau_0 {b}^{3} + \tau_0 \mu c_1 b
+ \tau_0 \mu b c_2 =  A \cos(b d_I) - B \sin(b d_I) \\
-{b}^{4}+ a_2 {b}^{2}-\tau_0 {b}^{2} c_1 - a_0
-\tau_0 {b}^{2} c_2 - \mu \tau_0 {b}^{2}+\tau_0 \mu c_1 c_2
= A \sin(b d_I) + B (\cos(b d_I))
\end{multline*}
with
$A = \tau_0 b (c_1 c_2 - {b}^{2}+ \mu (c_1 + c_2))$
and
$B = \tau_0 ( \mu c_1 c_2 - {b}^{2} (\mu + c_1 + c_2))$.
Adding up the squares of both equations, we obtain that
\begin{equation}
\label{eq:square}
{b}^{8}+ \alpha_3 {b}^{6}+\alpha_2 {b}^{4}+\alpha_1 {b}^{2}+\alpha_0  =0,
\end{equation}
where $\alpha_0 = a_0(a_0 -2\,\mu\,\tau_0\,c_1\,c_2)$,
\begin{equation*}
\begin{split}
\alpha_1 &= 2 \tau_0 (\mu (a_0 + a_2 c_1\,c_2 - a_1 (c_1+ c_2))
+ a_0 (c_1 + c_2) - a_1 c_1 c_2)-2 a_2 a_0 + a_1^{2},\\
\alpha_2 &= 2 \tau_0 \left( \mu (a_3(c_1 + c_2) - a_2 - c_1 c_2)
- a_2 (c_1 + c_2) + a_3 c_1 c_2 + a_1 \right) +2 a_0 + a_2^{2}  \\
& \hspace{4mm}  -2 a_3 a_1,\\[-0.2mm]
\alpha_3 &= 2 \tau_0 (\mu+c_1+c_2)+{a_3}^{2} -2(a_3 \tau_0+a_2).
\end{split}
\end{equation*}
Let $z = b^2$. Then (\ref{eq:square}) becomes
\begin{equation}
\label{eq:square2}
{z}^{4}+ \alpha_3 {z}^{3}+\alpha_2 {z}^{2}+\alpha_1 {z} +\alpha_0  =0.
\end{equation}
By the Routh--Hurwitz criterion, (\ref{eq:square2}) has no positive real roots
if $\alpha_i > 0$, $i = 0, \ldots, 3$, $\alpha_3 \alpha_2 > \alpha_1$,
and $\alpha_3 \alpha_2 \alpha_1 > \alpha_1^2 + \alpha_3^2 \alpha_0$.

For the parameter values of Table~\ref{table:parameters} and
$\beta = 40$ ($R_0 = 0.880827$), the Routh--Hurwitz criterion does not hold.
In fact, equation (\ref{eq:square2}) takes the form
\begin{equation}
\label{eq:square:param}
{z}^{4}+ 241.429794 {z}^{3}+ 31.065028 {z}^{2} -221.270089 {z} -0.037233 =0
\end{equation}
and we immediately see that the coefficient $\alpha_1 = -221.270089$ is not
positive. Moreover, the roots of equation (\ref{eq:square:param}) 
are approximately $0.89$, $-0.00017$, $-1.03$, $-241.30$, 
therefore there exists a positive imaginary root given by $0.89 i$. 
Thus, there exists at least one time
delay such that the disease free equilibrium is unstable.
We have just proved the following result.

\begin{lemma}
Let $R_0 < 1$. Then there exists at least
one positive time delay $d_I > 0$
such that the disease free equilibrium
$(N, 0, 0, 0)$ is unstable.
\end{lemma}

\begin{remark}
\label{obs1}
Observe that there may exist specific time delays for which the
disease free equilibrium $(N, 0, 0, 0)$ is locally asymptotically
stable when $R_0 < 1$. As we show next, this is the case for
the time delay $d_I = 0.1$. Indeed, consider the parameter values
of Table~\ref{table:parameters} and $\beta$ such that $R_0 < 1$. For example,
let $\beta = 40$, for which $R_0 = 0.880827$. The results are given
in Figure~\ref{fig:statevar:DFE}, where to compute the trajectories 
we have used the \textsc{Matlab} routine \texttt{dde23}, which 
solves delay differential equations with constant delays. 
For the theoretical results that underlie this solver 
we refer to \cite{MR1831793}.
\begin{figure}[ht!]
\centering
\subfloat[\footnotesize{State variables}]{\label{statevar:DFE}
\includegraphics[width=0.49\textwidth]{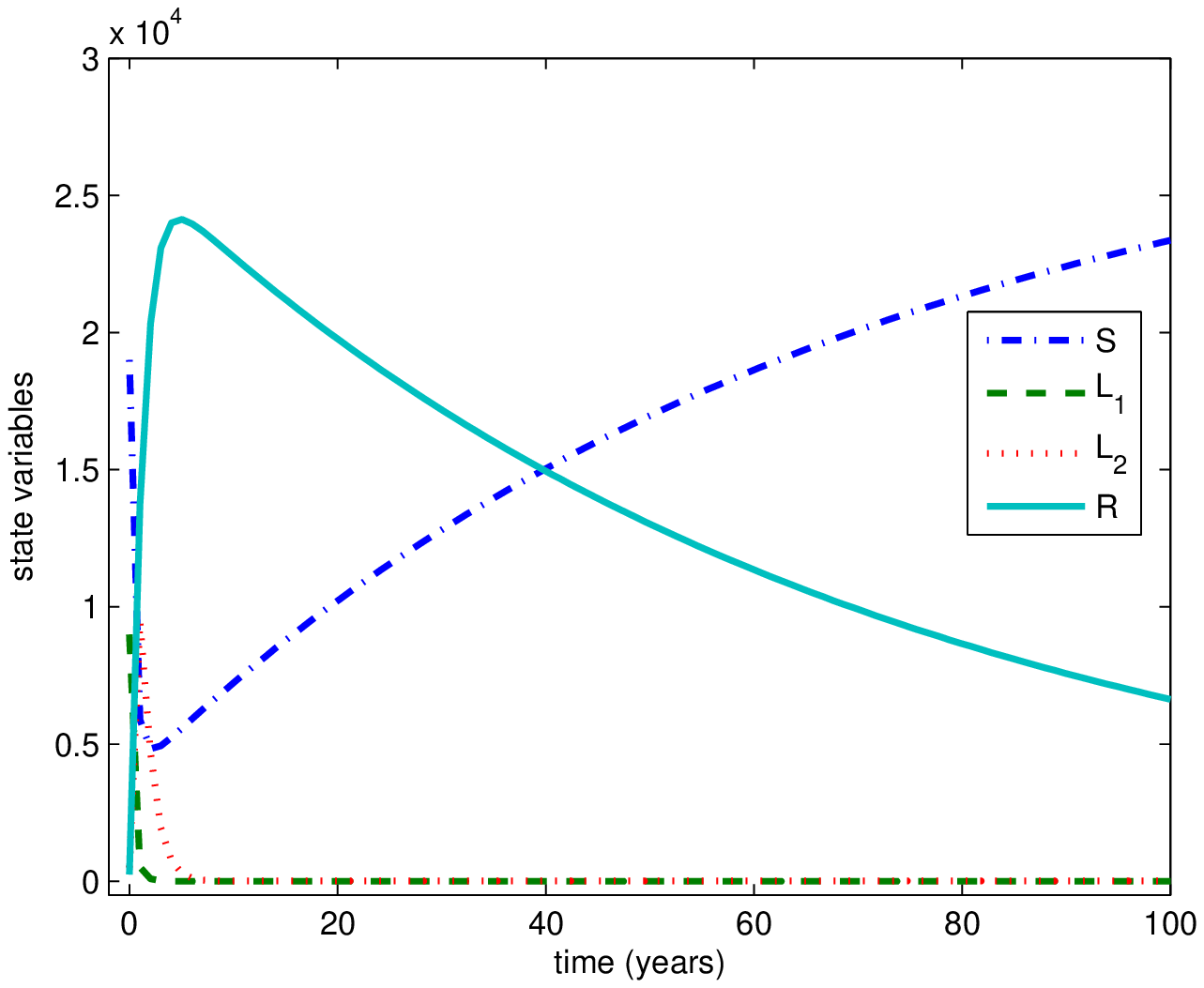}}
\subfloat[\footnotesize{Infected active TB ($T = 100$)}]{\label{I:DFE:tf100}
\includegraphics[width=0.49\textwidth]{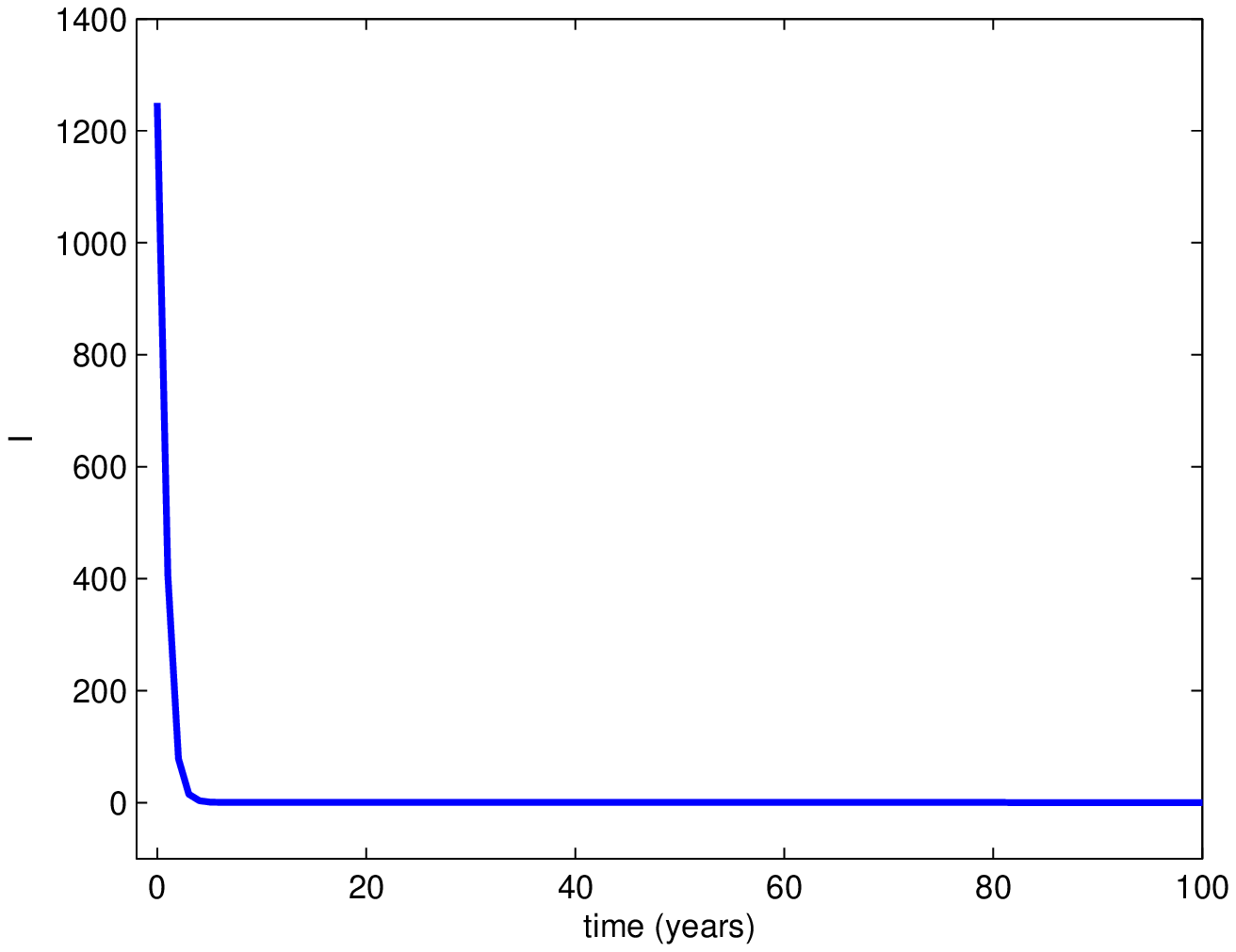}}
\caption{Disease free equilibrium with
basic reproduction number $R_0 = 0.88$ ($\beta = 40$, $d_I = 0.1$
and the other values from Table~\ref{table:parameters}).}
\label{fig:statevar:DFE}
\end{figure}
The characteristic equation (\ref{eq:caract}) takes the form $\chi(\lambda)=0$ with
\begin{equation*}
\begin{split}
\chi(\lambda) ={\lambda}^{4} &+ 17.057363\,{\lambda}^{3}+ 20.733305\,{\lambda}^{2}
+4.489748\,\lambda+ 0.048755\\
&+2\, \left( \lambda+{\frac {1}{70}}
\right)  \left( \lambda+{\frac {981}{70}} \right)  \left( \lambda
+1.014486 \right)  \left( {{\rm e}^{-\lambda\, 0.1}}-1 \right) \, .
\end{split}
\end{equation*}
The derivative $\chi'(\lambda)$ has exactly four zeros:
$\bar{\lambda}_1 = -21.408183$,
$\bar{\lambda}_2 = -12.680307$, $\bar{\lambda}_3 = -0.748206$,
and $\bar{\lambda}_4 = -0.151719$.
For all $\lambda \in ]-\infty, \bar{\lambda}_1[$ one has $\chi'(\lambda) > 0$
and since $\lim_{\lambda\rightarrow -\infty} \chi(\lambda) = - \infty$
and $\chi(\bar{\lambda}_1) > 0$, we conclude that there exists a unique
$\lambda_1$ in the interval $]-\infty, \bar{\lambda}_1[$ such that
$\chi(\lambda_1) = 0$. Analogously, we prove that there exists
exactly five roots of $\chi(\lambda)=0$:
$\lambda_1 = -23.481727$,
$\lambda_2 = -18.106597$,
$\lambda_3 =  -1.024343$,
$\lambda_4 =  -0.320880$,
$\lambda_5 =  -0.011482$.
We conclude that there are no positive real roots and that
the disease free equilibrium $(N, 0, 0, 0)$
is locally asymptotically stable for $d_I = 0.1$
and the considered parameter values.
\end{remark}

In this paper we assume that the time delay $d_I$ associated to the diagnosis
of active TB is equal to 0.1, that is, 36.5 days. This value makes sense from
the epidemiological point of view, since it fits in the intervals available
in the literature for the delay in the diagnosis of active TB. For instance,
in \cite{delay:diag:review:BMC:2009} the reported overall patient delay is
similar to the health system delay for diagnosis of active TB, 31.03 and
27.2 days, respectively. The average (median or mean) patient delay
and health system delay range from 4.9 to 162 days and 2 to 87 days,
respectively, in both low and middle income countries and high income countries.


\section{Stability of the endemic equilibrium}
\label{sec:4}

System (\ref{TBmodel:delay:state}) has an unique endemic equilibrium
such that $I(t) > 0$ for any $t > 0$. The analytic expression
is cumbersome and not useful for our purposes. Consider 
the parameter values from Table~\ref{table:parameters} 
with $\beta = 100$. Then the basic reproduction number 
is $R_0 = 2.202067$.
The endemic equilibrium $E^*$ is given by $I = 11.006448$,
$L_1 = 36.111397$, $L_2 = 402.155827$, $S = 8407.668384$.
The matrices $A_1$ and $A_2$ associated to the linearized system
(\ref{model:linear:delay:4eq}) at the endemic equilibrium $E^*$ 
are computed as 
\begin{equation*}
A_1 = \left[ \begin {array}{cccc}
- 0.050974&0&- 28.025561&0\\ \noalign{\medskip}
0.027516&- 14.023458& 45.970734&0\\ \noalign{\medskip}
- 0.000020& 0.599980&- 0.014306& 0.000180\\ \noalign{\medskip}
0& 11.400000&- 0.335130&- 1.023658
\end {array} \right]
\end{equation*}
and $A_2 = \textrm{diag}(0, 0, -2, 0)$.
The transcendental characteristic equation is given by
\begin{multline}
\label{eq:charac:EE}
{\lambda}^{4}+ 15.112395\,{\lambda}^{3}
- 12.243801\,{\lambda}^{2}- 28.331139\,\lambda - 0.966336\\
+ \left( 30.196179\,{\lambda}^{2}+ 30.244462\,
\lambda+ 2 \,{\lambda}^{3}+ 1.463482\right) {{\rm e}^{-\lambda\,
{\it d_I}}} = 0 \, .
\end{multline}
When $d_I = 0$, we have the following characteristic equation:
\begin{equation}
\label{eq:charac:EE:nuI:0}
{\lambda}^{4}+ 17.112395\,{\lambda}^{3}+ 17.952378\,{\lambda}^{2}
+ 1.913323\,\lambda+ 0.497146 = 0 \, .
\end{equation}
The roots of (\ref{eq:charac:EE:nuI:0}) are
$-1.029896$, $-15.997555$,
$-0.042472 -0.168435 \, i$, $-0.042472 +0.168435 \, i$.
All the roots of (\ref{eq:charac:EE:nuI:0}) have negative real part,
thus the endemic equilibrium $E^*$ is asymptotical stable.
Consider now the case $d_I > 0$ and suppose that (\ref{eq:charac:EE})
has a purely imaginary root $b\, i$, with $b > 0$. Separating
the real and imaginary parts in (\ref{eq:charac:EE}), we have
\begin{equation}
\label{eq:charac:EE:nuI:pos}
{b}^{8}+ 281.828573\,{b}^{6}- 51.906667\,{b}^{4}- 1.236501\,{b}^{2}
- 0.000246 = 0 \, .
\end{equation}
It is easy to verify that $b = 0.453220$ is a root of equation
(\ref{eq:charac:EE:nuI:pos}). Thus, by Rouch\'{e}'s theorem, there exists
at least a time delay $d_I > 0$ such that the endemic equilibrium $E^*$
is unstable. In the specific case $d_I = 0.1$,
the characteristic equation is given by
\begin{multline}
\label{eq:carac:EE:nuI:01}
{\lambda}^{4}+ 15.112395\,{\lambda}^{3}
-12.243801\,{\lambda}^{2} - 28.331139\,\lambda - 0.966336 \\
+ \left( 2\,{\lambda}^{3}+  30.196179\,{\lambda}^{2}+ 30.244462\,
\lambda+  1.463482 \right) {{\rm e}^{-0.1\,\lambda}}= 0 \, .
\end{multline}
Similarly to Remark~\ref{obs1}, it follows from Bolzano's theorem
and the monotonicity of the characteristic function associated
to (\ref{eq:carac:EE:nuI:01}) that all roots of equation
(\ref{eq:carac:EE:nuI:01}) have a negative real part.
Therefore, the endemic equilibrium $E^*$ is locally asymptotically stable
for $d_I = 0.1$ and $R_0 > 1$.
\begin{table}
\centering
\begin{tabular}{|l|l|l|}
\hline
{\small{Symbol}} & {\small{Description}}  & {\small{Value}} \\
\hline
{\small{$\beta$}} & {\small{Transmission coefficient}}  & {\small{$\in [50,150]$}}\\
{\small{$\mu$}} & {\small{Death and birth rate}}  & {\small{$1/70 \, yr^{-1}$}}\\
{\small{$\delta$}} & {\small{Rate at which individuals leave $L_1$}}
& {\small{$12 \, yr^{-1}$}}\\
{\small{$\phi$}} & {\small{Proportion of individuals going to $I$}}
& {\small{$0.05$}}\\
{\small{$\omega$}}
& {\small{Endogenous reactivation rate for persistent latent infections}}
& {\small{$0.0002 \, yr^{-1}$}}\\
{\small{$\omega_R$}}
& {\small{Endogenous reactivation rate for treated individuals}}
&{\small{$0.00002 \, yr^{-1}$}}\\
{\small{$\sigma$}}
& {\small{Factor reducing the risk of infection as a result of acquired}} & \\
& {\small{immunity to a previous infection for $L_2$}} & {\small{$0.25$}} \\
{\small{$\sigma_R$}} & {\small{Rate of exogenous reinfection of treated patients}}
& {\small{0.25}} \\
{\small{$\tau_0$}} & {\small{Rate of recovery under treatment of active TB}}
&  {\small{$2 \, yr^{-1}$}}\\
{\small{$\tau_1$}}
& {\small{Rate of recovery under treatment of early latent individuals $L_1$}}
&  {\small{$2 \, yr^{-1}$}}\\
{\small{$\tau_2$}}
& {\small{Rate of recovery under treatment of persistent latent individuals $L_2$}}
& {\small{$1 \, yr^{-1}$}}\\
{\small{$N$}} & {\small{Total population}} & {\small{$30,000$}} \\
{\small{$T$}} & {\small{Total simulation duration}} & {\small{$5$ $yr$}} \\
{\small{$\epsilon_1$}}
& {\small{Efficacy of treatment of early latent $L_1$}} & {\small{$0.5$}} \\
{\small{$\epsilon_2$}}
& {\small{Efficacy of treatment of persistent latent TB $L_2$}}
& {\small{$0.5$}}\\ \hline
\end{tabular}
\caption{Parameter values.}
\label{table:parameters}
\end{table}


\section{Optimal control of a tuberculosis model with state and control delays}
\label{sec:5}

We now consider the TB model \eqref{TBmodel:delay:state} with a time delay
in the state variable $I(t)$ and introduce two control functions
$u_1(\cdot)$ and $u_2(\cdot)$ and two real positive model constants
$\epsilon_1$ and $\epsilon_2$. The control $u_1$ represents the effort
on early detection and treatment of recently infected individuals $L_1$
and the control $u_2$ represents the application of chemotherapy
or post-exposure vaccine to persistent latent individuals $L_2$.
The parameters $\epsilon_i \in (0, 1)$, $i=1, 2$,
measure the effectiveness of the controls $u_i$, $i=1, 2$, respectively,
\textrm{i.e.}, these parameters measure the efficacy of post-exposure
interventions for early and persistent latent TB individuals, respectively.
Since after TB infection the human immune system can take from 2 to 8 weeks
to react, it takes at least the same time to detect  the infection  by the medical
test \cite{url:center:diseases}. Hence, we introduce a time delay $d_{u_1}$
in the control $u_1$ which represents the delay in the diagnosis of latent TB
and commencement of latent TB treatment. The prophylactic treatment
of persistent latent individuals may also suffer from a delay due to personal
reasons of the patient, who may be resistant to treatment having spent more
than two years with latent infection without passing to active disease, or the delay
may be caused by priorities given to early latent and active infectious individuals
from the health care system. Based on these facts, for numerical simulations
we shall consider the following delays with bounds on the control delays:
\begin{equation}
\label{delays}
d_I = 0.1,  \quad d_{u_1}, \, d_{u_2} \in [0.05,0.2]\,.
\end{equation}
The resulting model is given by the following system
of nonlinear ordinary delay differential equations:
\begin{equation}
\label{TBmodel:controls:delays}
\hspace{-3mm}
\begin{cases}
\dot{S}(t) = \mu N - \frac{\beta}{N} I(t) S(t) - \mu S(t),\\[0.3mm]
\dot{L_1}(t) = \frac{\beta}{N} I(t)\left( S(t)
+ \sigma L_2(t) + \sigma_R R(t)\right) - \left(\delta + \tau_1
+ \epsilon_1 u_1(t-d_{u_1}) + \mu\right)L_1(t),\\[0.3mm]
\dot{I}(t) = \phi \delta L_1(t) + \omega L_2(t) + \omega_R R(t)
- \tau_0 I(t - d_I) - \mu I(t),\\[0.3mm]
\dot{L_2}(t) = (1 - \phi) \delta L_1(t) - \sigma \frac{\beta}{N} I(t) L_2(t)
- (\omega + \epsilon_2 u_2(t-d_{u_2}) + \tau_2 + \mu)L_2(t).
\end{cases}
\end{equation}
Recall that the recovered population is determined
by $R(t) = N - (S(t)+L_1(t)+I(t)+L_2(t))$,
which formally gives the equation
\begin{eqnarray*}
\dot{R}(t) &= & \tau_0 I(t - d_I) + (\tau_1 + \epsilon_1 u_1(t-d_{u_1})) L_1(t)
+ (\tau_2 + \epsilon_2 u_2(t-d_{u_2})) L_2(t)\\
&& - \sigma_R \frac{\beta}{N} I(t) R(t) - (\omega_R + \mu)R(t) \, .
\end{eqnarray*}
Note, however, that this equation is  not needed in the subsequent optimal
control computations. We prescribe the following initial conditions for the
state variables $(S,L_1,L_2)$ and, due to the delays, initial functions
for the state variable $I$ and controls $u_1$ and $u_2$:
\begin{equation}
\label{initial-condition}
 \hspace{0mm}
\begin{array}{l}
S(0) = (76/120)N, \; L_1(0) = (36/120) N , \;
L_2(0) =(2/120)N, \; R(0) = (1/120)N , 
\\[0.5mm]
I(t) = (5/120)N \;\; \mbox{for} \; -d_I \leq t \leq 0,
\; u_k(t) =0  \;\; \mbox{for} \; -d_{u_k} \leq t < 0 \;\; (k=1,2).
\end{array}
\end{equation}
In the case $d_{u_1} = d_{u_2} = 0 $ of no control delays, the last condition
is void. The following box constraints are imposed on the control variables:
\begin{equation}
\label{control-constraints}
0 \leq u_k(t) \leq 1 \quad \forall \; t \in [0,T] \quad (k=1,2).
\end{equation}
Let us denote the state and control variable of the control system
\eqref{TBmodel:controls:delays}, respectively, by
$x=(S,L_1,I,L_2) \in \mathbb{R}^4$ and $u =(u_1,u_2) \in \mathbb{R}^2$.
We shall consider two types of objectives: either the $L^1$--type objective
\begin{equation}
\label{J1-objective}
J_1(x,u)  = \int_0^{T} (I(t) + L_2(t) + W_1\,u_1(t) + W_2\,u_2(t))\,dt \,,
\end{equation}
which is linear in the control variable $u$, or the $L^2$--type objective
\begin{equation}
\label{J2-objective}
J_2(x,u)  = \int_0^{T} (I(t) + L_2(t) + W_1\,u_1^2(t) + W_2\,u_2^2(t))\,dt \,,
\end{equation}
which is quadratic in the control variable. In both objectives, $W_1 > 0$,
$W_2 > 0$ are appropriate weights to be chosen later. $L^2$-type functionals
like \eqref{J2-objective} are often used in economics to describe, e.g.,
productions costs, but are not appropriate in a biological framework; cf.
the remarks in \cite{Schaettler-Ledzewicz-Maurer}. The $L^1$ functional
$J_1(x,u)$ incorporates  the total amount of drug used as a penalty
and thus appears to be more realistic.  For that reason, we shall mainly focus
on the functional $J_1(x,u)$.

The optimal control problem then is defined as follows: determine a control
function $u=(u_1,u_2) \in L^1([0,T],\mathbb{R}^2)$ that {\it minimizes} either
the cost functional $J_1(x,u)$ in \eqref{J1-objective} or $J_2(x,u)$
in \eqref{J2-objective} subject to the dynamic constraints
\eqref{TBmodel:controls:delays}, initial conditions \eqref{initial-condition}
and control constraints \eqref{control-constraints}. Necessary optimality
conditions for optimal control problems with multiple time delays in control
and state variables may be found, e.g., in \cite{Goellmann-Maurer-14}.
Here, we discuss the {\it Maximum Principle} in order to display the controls
and the switching functions in a convenient way.
To define the Hamiltonian for the delayed control problem,
we introduce the delayed state variable $y_3(t) = x_3(t-d_I) = I(t-d_I)$ and
the delayed control variables $v_k(t) = u_k(t-d_{u_1})$, $k=1,2$.
Using the adjoint variable 
$\lambda = (\lambda_S,\lambda_{L_1},\lambda_I,\lambda_{L_2}) \in \mathbb{R}^4$,
the Hamiltonian for the objective $J_1$ and the control system
\eqref{TBmodel:controls:delays} is given by
\begin{equation*}
\begin{array}{l}
H(x,y_3,\lambda,u_1,u_2,v_1,v_2)
= -(I + L_2+W_1 u_1 + W_2 u_2)
+ \lambda_S (\mu N - \frac{\beta}{N} I S - \mu S )\\[1mm]
\quad + \lambda_{L_1} ( \frac{\beta}{N} ( S + \sigma L_2 + \sigma_R R)
- (\delta + \tau_1 + \epsilon_1 v_1 + \mu )L_1)\\[1mm]
\quad + \lambda_I (\phi \delta L_1 + \omega L_2 + \omega_R R
- \tau_0 y_3 - \mu I ) \\[1mm]
\quad + \lambda_{L_2} ((1 - \phi) \delta L_1- \sigma \frac{\beta}{N} I L_2
- (\omega + \epsilon_2 v_2 + \tau_2 + \mu)L_2 ) .
\end{array}
\end{equation*}
We obtain the adjoint equations
$\dot{\lambda}_S(t) = - H_S [t]$,
$\dot{\lambda}_{L_1}(t) = - H_{L_1} [t]$,
$\dot{\lambda}_{L_2}(t) = - H_{L_2} [t]$, and
$\dot{\lambda}_I(t) =  - H_I [t] + \chi_{\,[0,T-d_I]}\, H_{y_3}[t+d_I]$,
where subscripts denote partial derivatives and $\chi_{\,[0,T-d_I]}$
is the characteristic function in the interval $[0,T-d_I]$
\cite{Goellmann-Maurer-14}. Note that only the equation for
$\dot{\lambda}_I(t)$ contains the {\it advanced} time $t+d_I$.
Since the terminal state $x(T)$ is free, the transversality conditions are
\begin{equation}
\label{transversality}
\lambda_S(T)= \lambda_{L_1}(T)= \lambda_I(T)= \lambda_{L_2}(T)= 0 .
\end{equation}
To characterize the optimal controls $u_1$ and $u_2$, we introduce
the following {\it switching functions} for $k=1,2$:
\begin{equation}
\label{eq:switch}
\begin{array}{rl}
\phi_k(t) &= \; H_{u_k}[t] + \chi_{\,[0,T-d_{u_k}]}(t+d_{u_k})\, 
H_{v_k}[t+d_{u_k}]\\[1mm]
& = \;
\left\{
\begin{array}{lcl}
-W_k - \epsilon_k \lambda_{L_k}(t+d_{u_k})  L_k(t+d_{u_k})
& \mbox{for} & 0 \leq t \leq T-d_{u_k},\\[1mm]
-W_k & \mbox{for} &  T-d_{u_k} \leq t \leq T.
\end{array}
\right.
\end{array}
\end{equation}
Then the maximum condition for the optimal controls $u_1(t)$, $u_2(t)$
is  equivalent to the maximization $\,\phi_k(t)u_k(t)
= \max_{\, 0 \leq u_k \leq 1} \; \phi_k(t) u_k$, $k=1,2$,
which gives the control law
\begin{equation}
\label{control-law}
u_k(t) = \left\{
\begin{array}{rcl}
1 &&\mbox{if} \quad \phi_k(t) > 0,  \\[1mm]
0 && \mbox{if} \quad \phi_k(t) < 0,  \\[1mm]
{\rm singular} &&  \mbox{if} \quad \phi_k(t) = 0
\;\; \mbox{on} \; I_s \subset [0,T],
\end{array}
\right.
\quad k=1,2.
\end{equation}
We do not discuss singular controls further, since both in the non-delayed
and the delayed control problem we did not find singular arcs. In view of the
transversality conditions \eqref{transversality}, the terminal value of the
switching function is $\phi_k(T) = -W_k$ for $k=1,2$. Hence, we may conclude
from the control law \eqref{control-law} that $u_1(t) = u_2(t)= 0$ holds
on a terminal interval.


\section{Numerical results for the non-delayed and the delayed optimal control problem}
\label{sec:6}

We choose the numerical approach ``First Discretize then Optimize''
to solve both the non-delayed and delayed optimal control problem.
The discretization of the control problem on a fine grid 
leads to a large-scale nonlinear programming problem (\textsc{NLP}) 
that can be conveniently formulated with the help 
of the Applied Modeling Programming Language \textsc{AMPL} \cite{Fourer}. 
\textsc{AMPL} can be linked to several powerful optimization solvers. 
We use the Interior-Point optimization solver \textsc{IPOPT}
developed by W\"achter and Biegler \cite{Waechter-Biegler}.
Details of discretization methods for delayed control problems may 
be found in \cite{Goellmann-Maurer-14}. The subsequent computations 
for the terminal time $t_f=5$ have been performed with 
$N=2500$ to $N=5000$ grid points using the trapezoidal rule 
as integration method. Choosing the error tolerance $tol = 10^{-10}$ 
in \textsc{IPOPT}, we can expect that the state variables are correct 
up to $7$ or $8$ decimal digits. Since Lagrange multipliers are computed 
a posteriori in \textsc{IPOPT}, we cannot expect more than 6 correct 
decimal digits in the adjoint variables. 

Also, the control package \textsc{NUDOCCCS} developed 
by B\"uskens \cite{Bueskens} (cf. also \cite{Bueskens-Maurer}) 
provides a highly efficient method for solving the discretized control problem,
because it allows to implement higher order integration methods. 
However, so far \textsc{NUDOCCCS} can only be implemented for 
non-delayed control problems. For the non-delayed TB control problem, 
we obtained only bang-bang controls. An important feature of \textsc{NUDOCCCS} 
is the fact that it provides an efficient method for optimizing the switching 
times of bang-bang controls using the arc-parametrization method 
\cite{Maurer-etal-OCAM}. This approach is called the  
Induced Optimization Problem (\textsc{IOP}) for bang-bang controls. 
\textsc{NUDOCCCS} then allows for a check of second-order
sufficient conditions of the \textsc{IOP}, whereby the second-order 
sufficient conditions for bang-bang controls can be verified 
with high accuracy; cf. \cite{Maurer-etal-OCAM,Osmolovskii-Maurer-BOOK}.


\subsection{Optimal control solution of the non-delayed TB model}

First, we consider the optimal control of non-delayed TB model, where formally
we put $d_I=d_{u_1}=d_{u_2}=0$. The numerical solutions serve as reference
solutions, which later will be compared with the solutions to the delayed
control problem. We choose the weights $W_1 = W_2 = 50$ in the objective
(\ref{J1-objective}) and the parameter $\beta = 100$ ($R_0 = 2.2$) in
Table~\ref{table:parameters}. The discretization approach shows that controls
$u_k(t)$ are bang-bang and with only one switch at $t_k$, $k=1,2$:
\begin{equation}
\label{optimal-control-nondelayed}
u_k(t) =
\left\{ \begin{array}{rcl}
1 & \quad \mbox{for} & 0 \leq t \leq t_k,  \\
0 & \quad \mbox{for} & t_k <  t \leq T,
\end{array} 
\right.
\quad k=1,2.
\end{equation}
To obtain a refined solution, we solve the \textsc{IOP} with respect to the
switching times $t_1$ and $t_2$ using the arc-parametrization method
\cite{Maurer-etal-OCAM} and  the code \textsc{NUDOCCCS} \cite{Bueskens}.
We get the following numerical results:
\begin{equation}
\label{results-nondelayed}
\hspace{-3mm}
\begin{array}{rclrclrcl}
J_1(x,u)&=& 28390.73, & \quad t_1 &=& 3.677250, & \quad t_2 &=&  4.866993,\\
S(T) &=&  1034.634 , & \quad L_1(T) &=& 53.59586 , & \quad I(T) &=& 25.89556,\\
L_2(T) &=& 780.7667  , & \quad R(T) &=&  28105.11 .
\end{array}
\end{equation}
The initial value of the adjoint variable 
$\lambda = (\lambda_s,\lambda_{L_1},\lambda_I,\lambda_{L_2})$
is computed as
$$
\lambda(0) = (0.376159, 0.452761, 4.03059, 0.394839).
$$
The control and state trajectories are displayed in Figure~\ref{label:of:fig:2}.
\begin{figure}[ht!]
\centering
\includegraphics[width=0.31\textwidth]{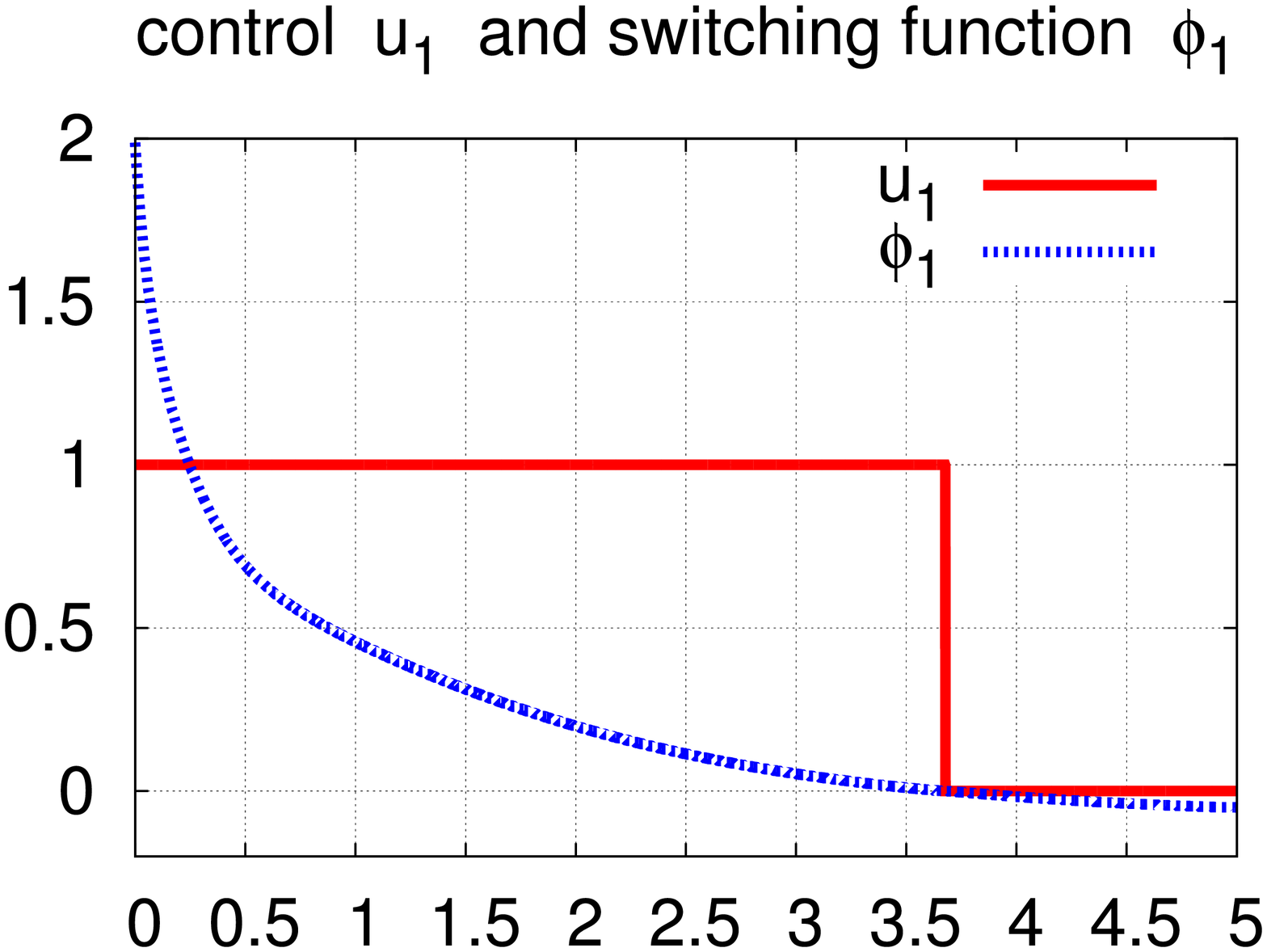}
\hspace{2mm} \includegraphics[width=0.31\textwidth]{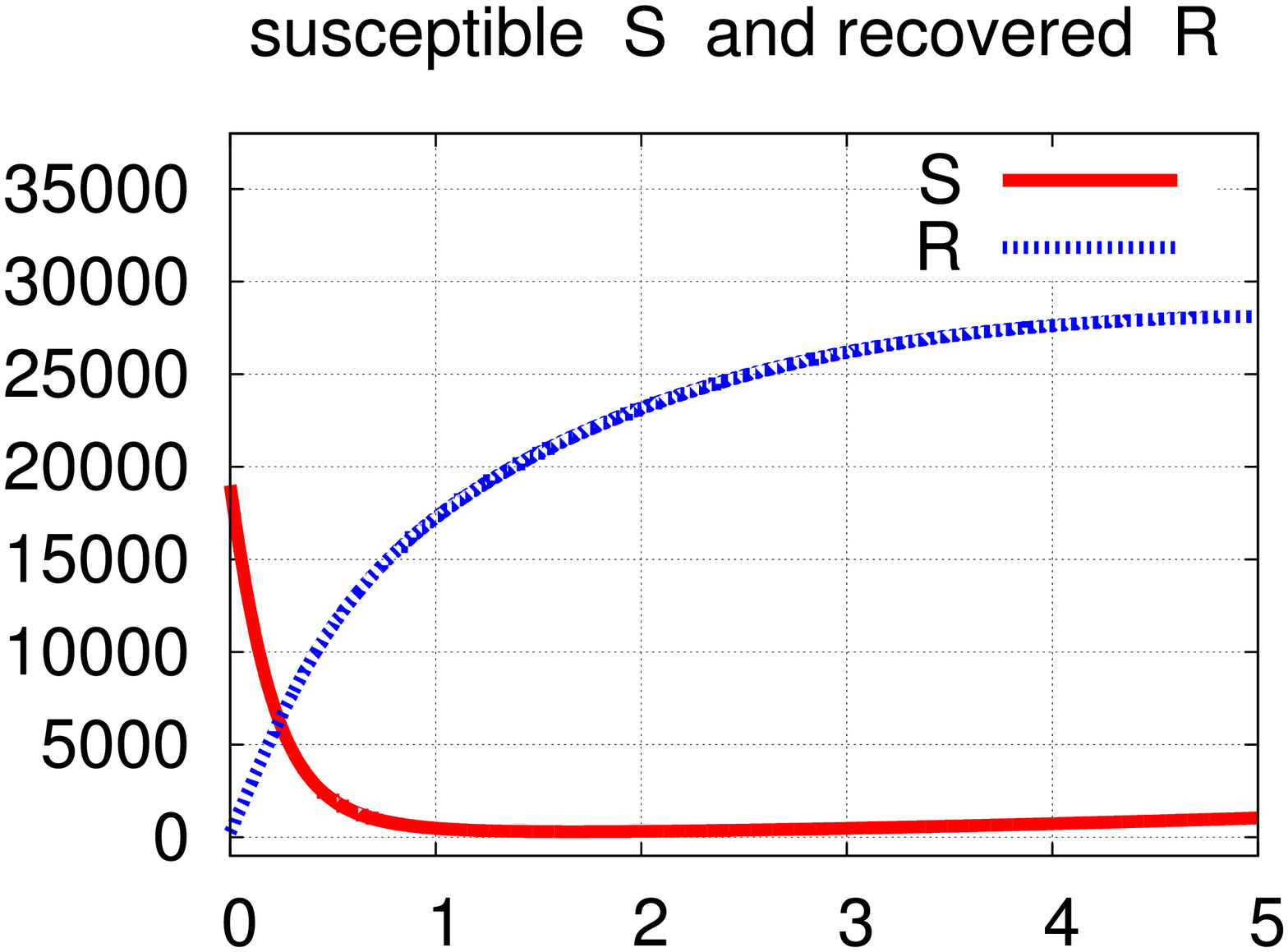}
\hspace{2mm}\includegraphics[width=0.31\textwidth]{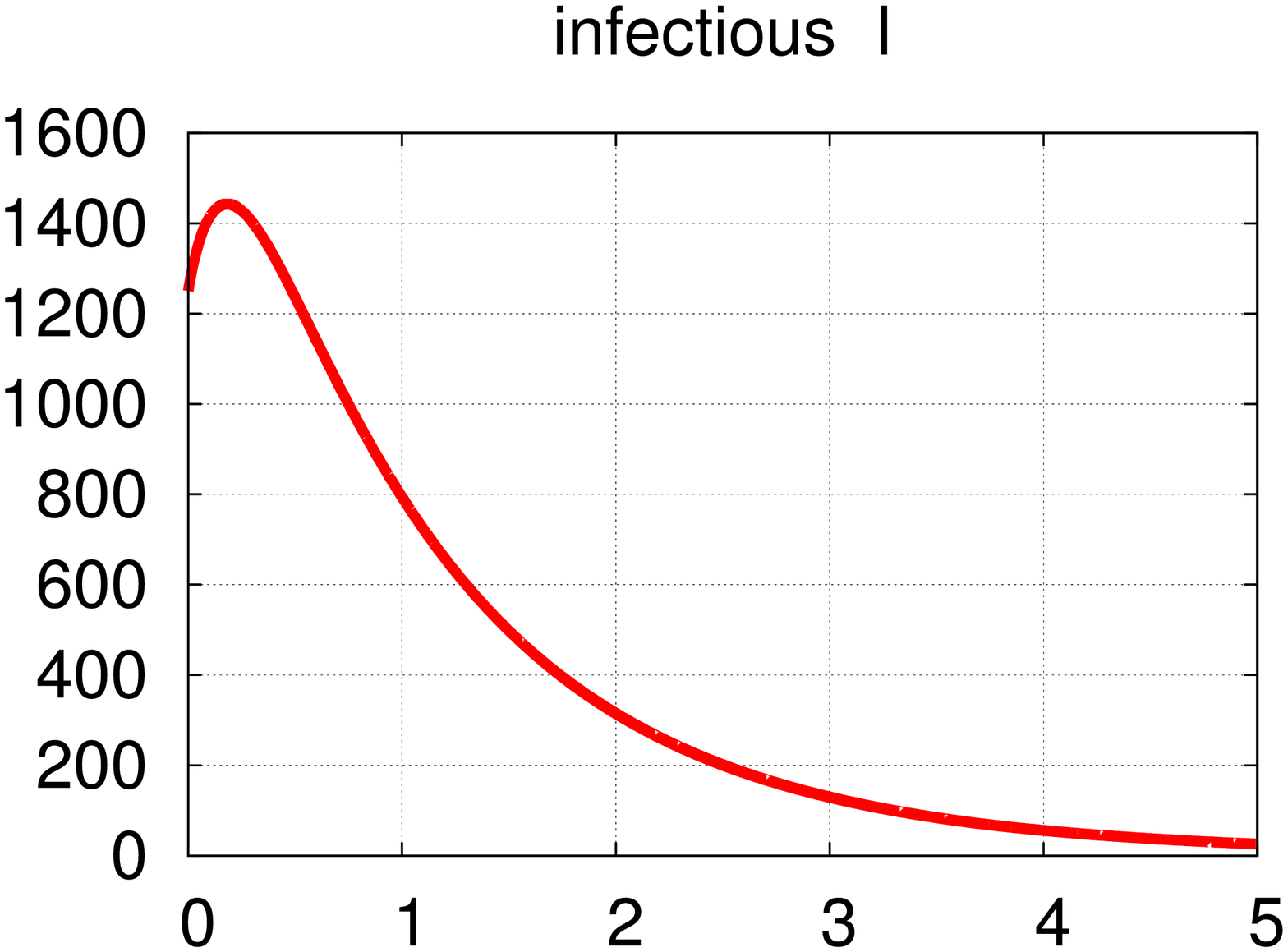}\\[0.2cm]
\includegraphics[width=0.31\textwidth]{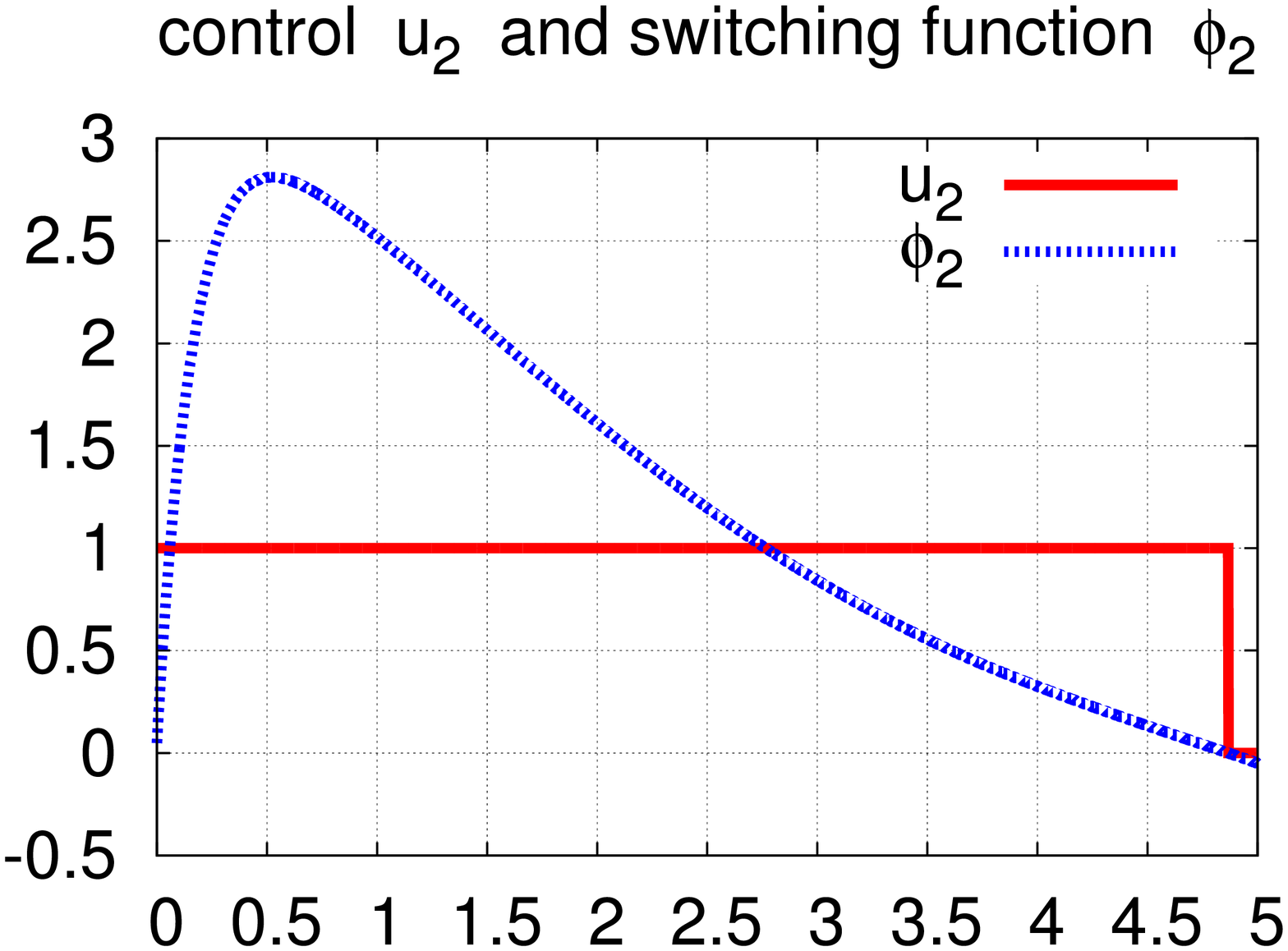}
\hspace{2mm} \includegraphics[width=0.31\textwidth]{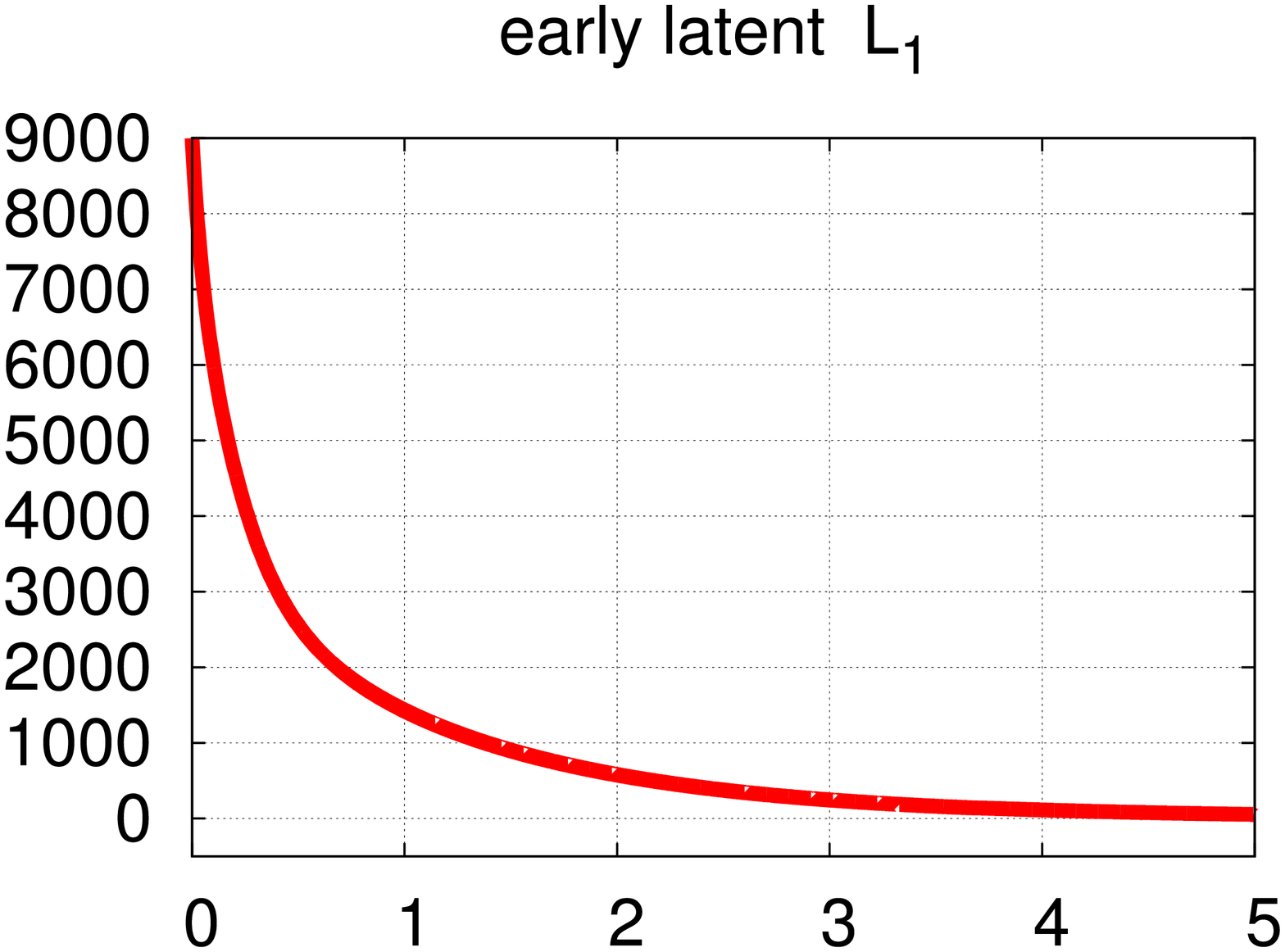}
\hspace{2mm}\includegraphics[width=0.31\textwidth]{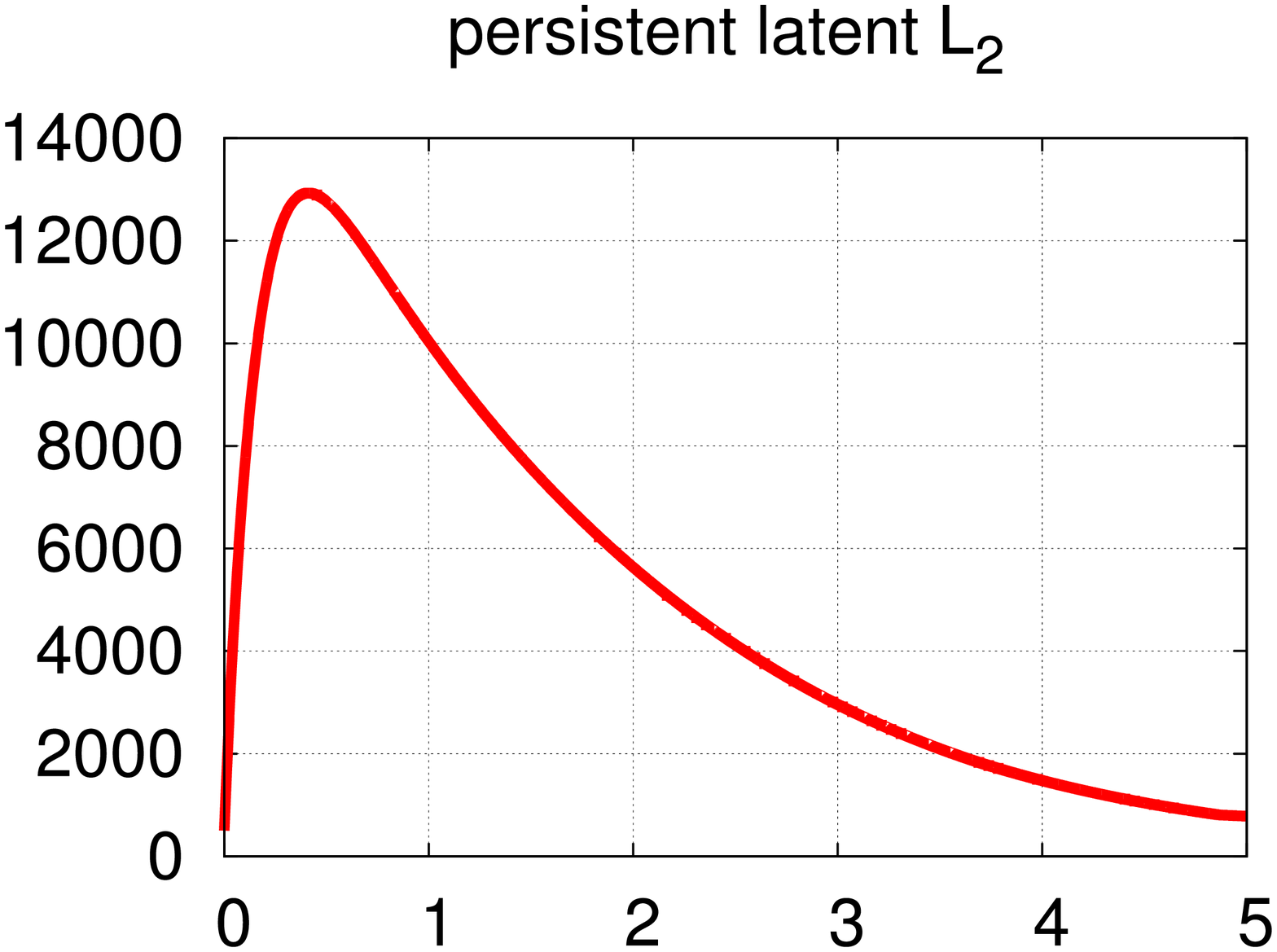}
\caption{Optimal control and state variables of the non-delayed TB model
with $L^1$ objective \eqref{J1-objective} and weights $W_1=W_2=50$.
{\it Top row}: (a) control $u_1$ \eqref{optimal-control-nondelayed}
and (scaled) switching function $\phi_1$ \eqref{eq:switch}
satisfying the control law \eqref{control-law} for $k=1$,
(b) susceptible individuals $S$ and recovered individuals $R$,
(c) infectious individuals $I$.
{\it Bottom row}: (a) control $u_2$ \eqref{optimal-control-nondelayed}
and (scaled) switching function $\phi_2$ \eqref{eq:switch}
satisfying the control law \eqref{control-law} for $k=2$,
(b) early latent  $L_1$, (c) persistent latent $L_2$.}
\label{label:of:fig:2}
\end{figure}
The Hessian $HessL$ of the Lagrangian for the \textsc{IOP} is positive definite:
$HessL = \left(
\begin{array}{ll}
453.98    &    387.42 \\
387.42    &    385.69
\end{array} \right) > 0$.
It can be seen from Figure~\ref{label:of:fig:2}, top row (a) and bottom row (a),
that the switching functions satisfy the strict bang-bang property
(cf. \cite{Maurer-etal-OCAM,Osmolovskii-Maurer-BOOK}) corresponding
to the Maximum Principle: $\phi_k(t) > 0$ for $0 \leq t < t_k$,
$\dot\phi_k(t_k) <0$, $\phi_k(t) < 0$ for $t_k < t \leq T \,(k=1,2)$.
Hence, the bang-bang controls \eqref{optimal-control-nondelayed} characterized
by the data \eqref{results-nondelayed} satisfies the second-order
sufficient conditions in \cite[Chap.~7]{Osmolovskii-Maurer-BOOK}
and thus provides a strict strong minimum. Figure~\ref{label:of:fig:3} displays
the comparison of the optimal controls for the functionals $J_1(x,u)$
 \eqref{J1-objective} and $J_2(x,u)$ \eqref{J2-objective}. The state
variables are nearly identical, since the control variables differ only
a terminal interval. Also the objective values are very close:
$J_1(x,u) =  28390.73$, $J_2(x,u) = 28382.37$. Note that the controls
for $J_2(x,u)$ are {\it continuous}, since the strict Legendre--Clebsch
condition holds and the Hamiltonian has a unique minimum. The proof, that
second-order sufficient conditions (SSC) are satisfied for the controls
corresponding to $J_2(x,u)$, is quite elaborated since it requires to
test whether an associated matrix Riccati equation has a bounded solution;
cf. \cite{Osmolovskii-Maurer-BOOK}.
\begin{figure}[ht!]
\centering
\includegraphics[width=0.40\textwidth]{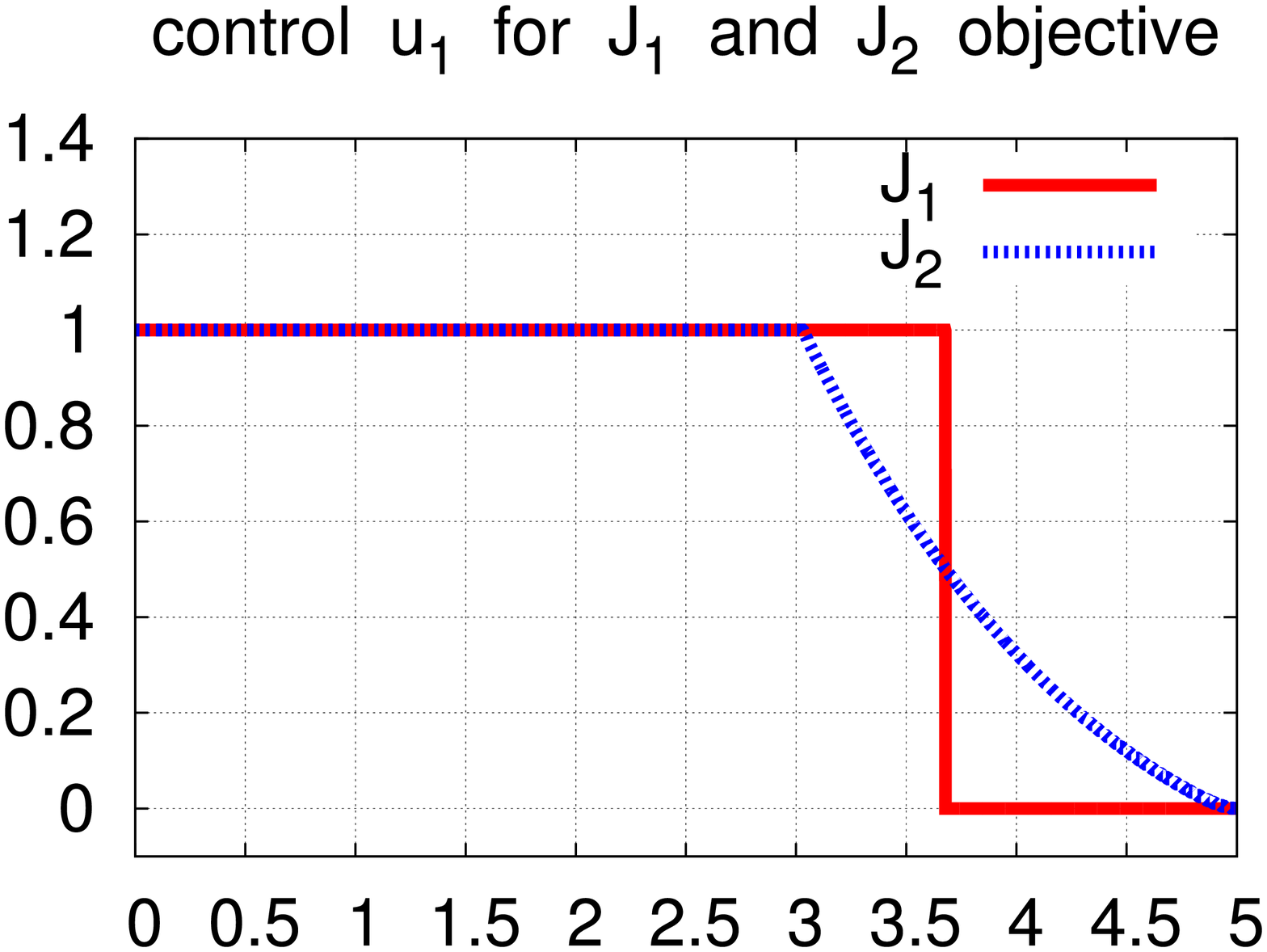}
\hspace{10mm}
\includegraphics[width=0.40\textwidth]{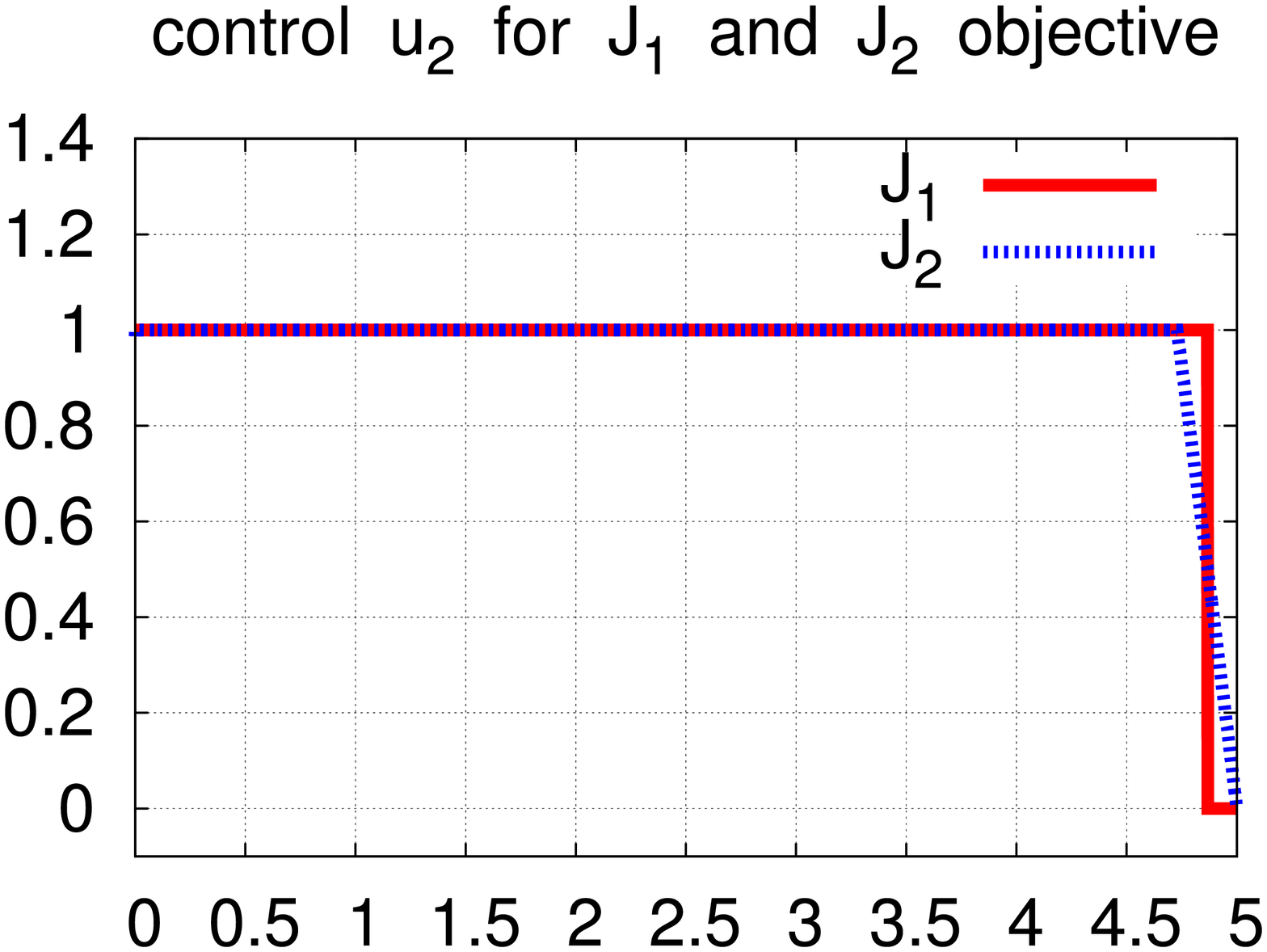}
\caption{Comparison of controls $u_1$ and $u_2$ for the $L^1$-type objective
\eqref{J1-objective} and $L^2$-type objective \eqref{J2-objective}
with weights $W_1=W_2 =50$.}
\label{label:of:fig:3}
\end{figure}
When we increase the weights $W_1$ and $W_2$, the control $u_1$ stays to be
bang-bang with only one switching time which gets smaller, whereas the control
$u_2$ may have an additional zero arc at the beginning. E.g., for
$W_1=W_2 = 150$ we obtain the controls
\begin{equation*}
(u_1(t),u_2(t)) =
\left\{
\begin{array}{rcl}
(1,0) & \quad \mbox{for} & 0 \leq t \leq t_1, \\
(1,1) & \quad \mbox{for} & t_1 <  t \leq t_2,\\
(0,1) & \quad \mbox{for} & t_2 <  t \leq t_3,\\
(0,0) & \quad \mbox{for} & t_3 <  t \leq T,
\end{array} 
\right.
\quad k=1,2 .
\end{equation*}
The objective value and the switching times are computed as
$J_1(x,u) =  29175.97$, $t_1 = 0.00260$, $t_2 = 2.662$,
and $t_3 = 4.633$. The optimal controls
are shown in Figure~\ref{label:of:fig:4}.
\begin{figure}[ht!]
\centering
\includegraphics[width=0.40\textwidth]{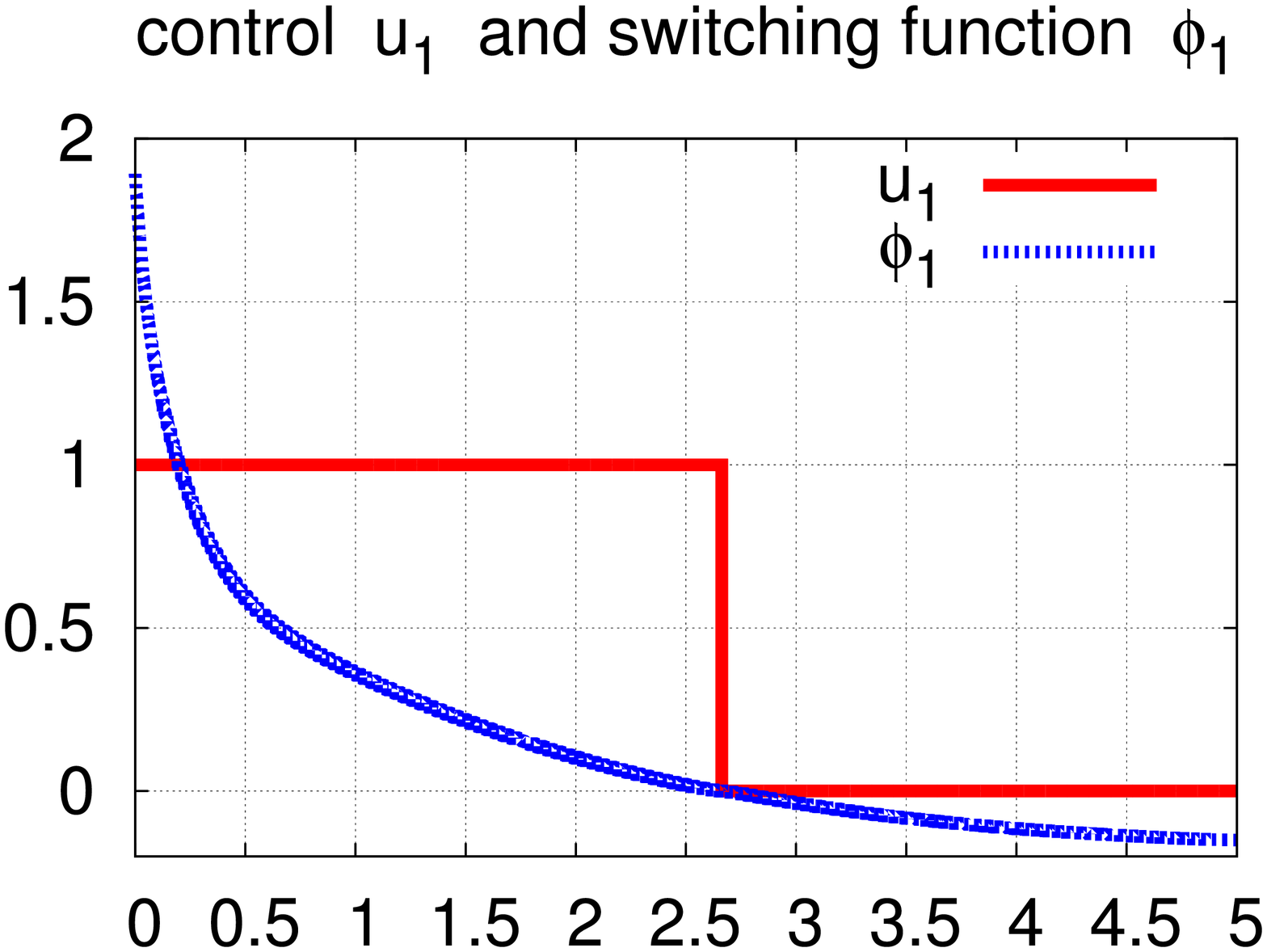}
\hspace{8mm}
\includegraphics[width=0.40\textwidth]{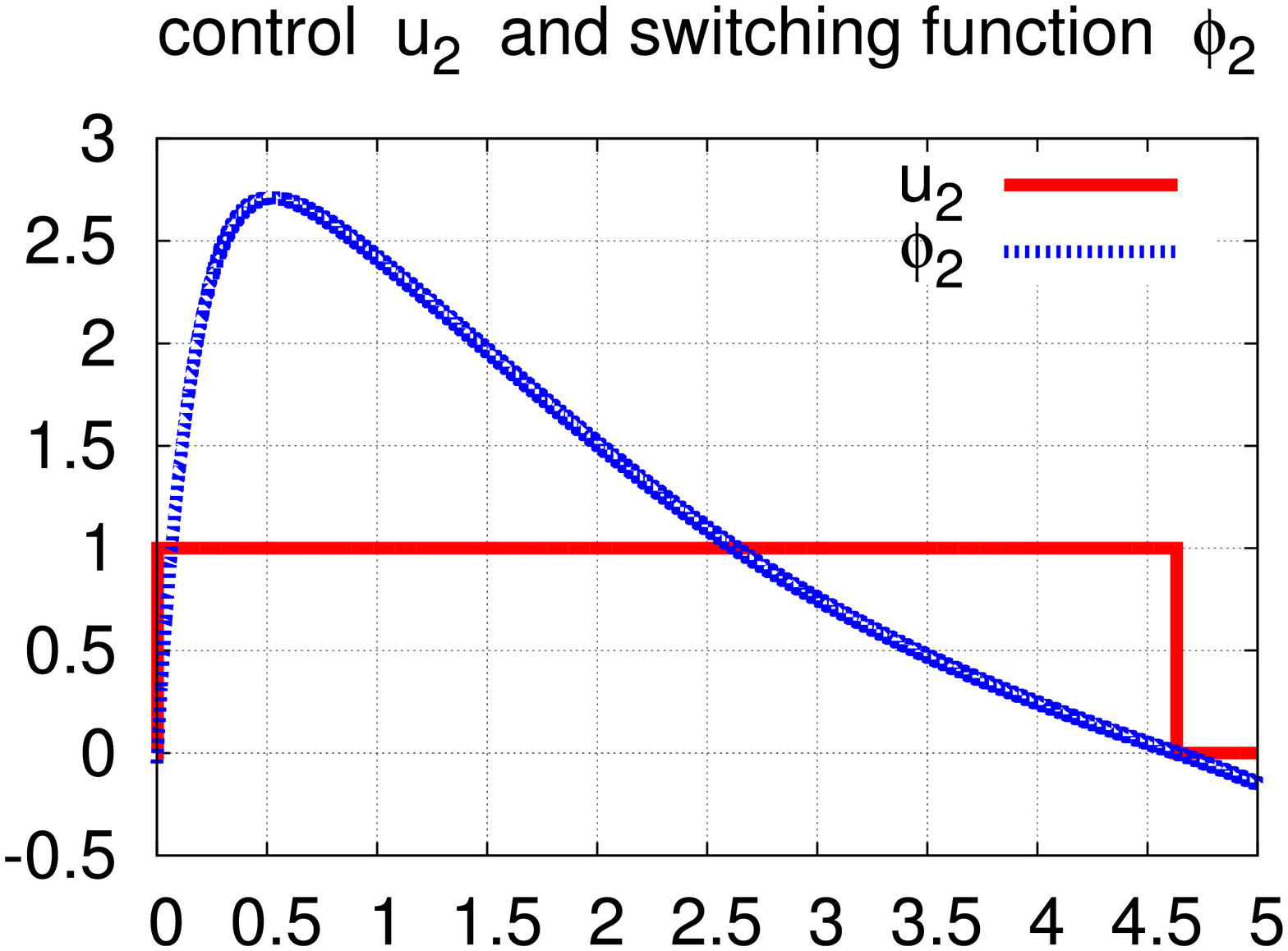}
\caption{Optimal controls $u_1$ and $u_2$ for the $L^1$-type objective
\eqref{J1-objective} with weights $W_1=W_2 =150$.}
\label{label:of:fig:4}
\end{figure}


\subsection{Optimal control solution of the TB model with control and state delays}

To see more distinctively the difference between delayed and non-delayed
solutions, we consider state and control delays with values at their upper
bounds in \eqref{delays}, that is, $d_I = 0.1$, $d_{u_1} = 0.2$,
and $d_{u_2} = 0.2$. Again, we choose the weights $W_1 = W_2 = 50$
in the objective (\ref{J1-objective}) and the parameter $\beta = 100$
in Table~\ref{table:parameters}, for which $R_0 = 2.2$. 
The discretization approach with $N=5000$
grid points and the trapezoidal rule as integration method
yields the following bang-bang controls with only one switch
as in the non-delayed case:
\begin{equation}
\label{eq:controlLawDelay}
u_k(t) =
\left\{
\begin{array}{rcl}
1 & \mbox{for} & 0 \leq t \leq t_k,  \\
0 & \mbox{for} & t_k <  t \leq T,
\end{array} 
\right.
\quad k=1,2.
\end{equation}
We obtain the numerical results $J_1(x,u)= 26784.60$, $t_1 = 3.108$,
$t_2 = 4.581$, $S(T) =  1234.598$, $L_1(T) = 24.93928$, $I(T) =  11.71451$,
$L_2(T) = 469.8865$, and $R(T) = 28258.86$. However, in contrast to the
non-delayed case, we are not able to optimize the switching times directly
because the time-transformation in the arc-parametrization method
\cite{Maurer-etal-OCAM} cannot be applied to the delayed problem. The initial
value of the adjoint variable is computed as
$\lambda(0) = (0.3789, 0.4682, 3.6412, 0.4263)$. When comparing the results
for the delayed problem with those in \eqref{results-nondelayed} for the
non-delayed problem, we notice the surprising fact that the terminal values
$L_1(T)$, $I(T)$, $L_2(T)$ and the switching times $t_1$, $t_2$ are
significantly smaller in the delayed problem on the expense that the terminal
value $S(T)$ is significantly higher.
\begin{figure}[ht!]
\centering
\includegraphics[width=0.32\textwidth]{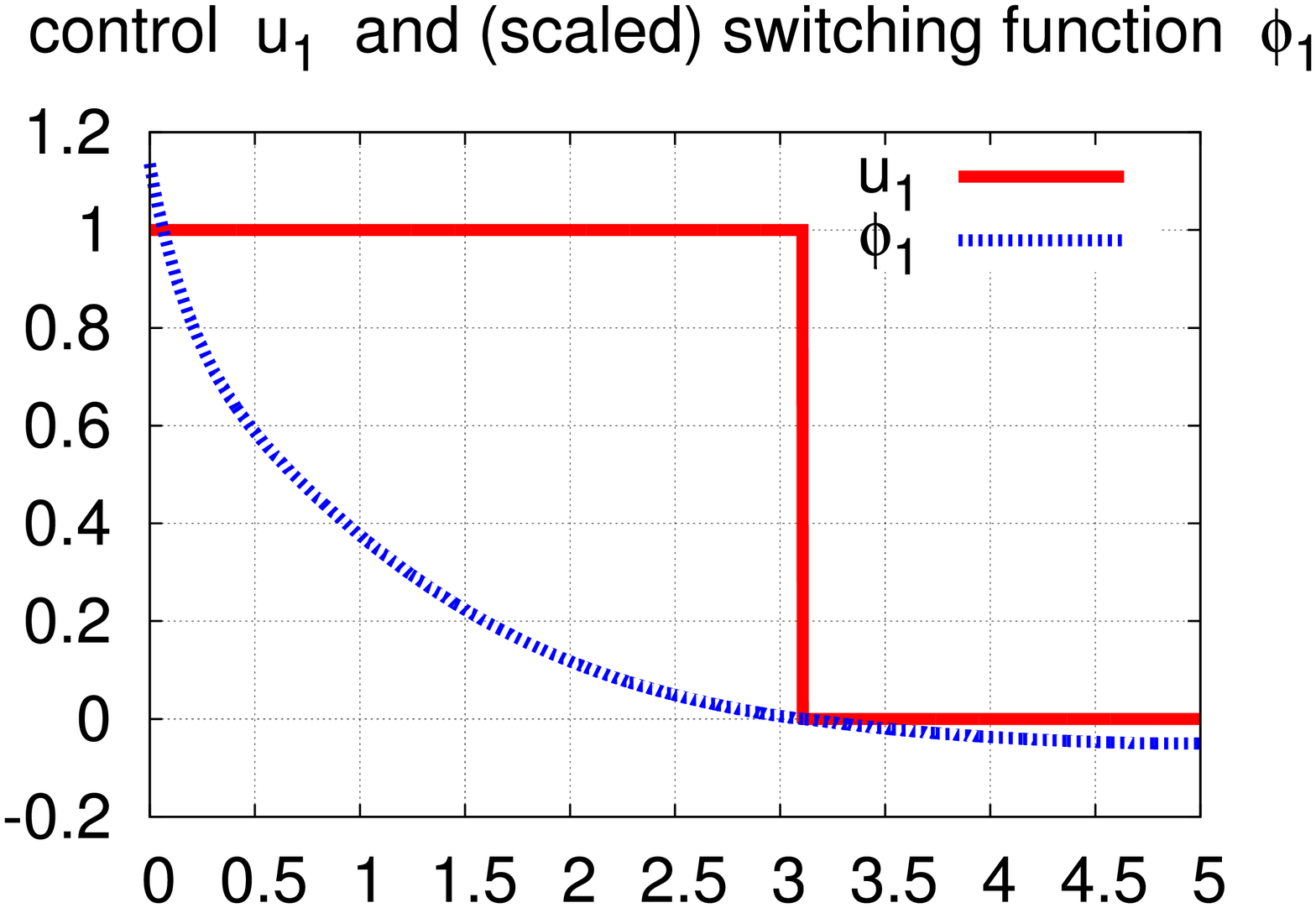}
\hspace{2mm}
\includegraphics[width=0.29\textwidth]{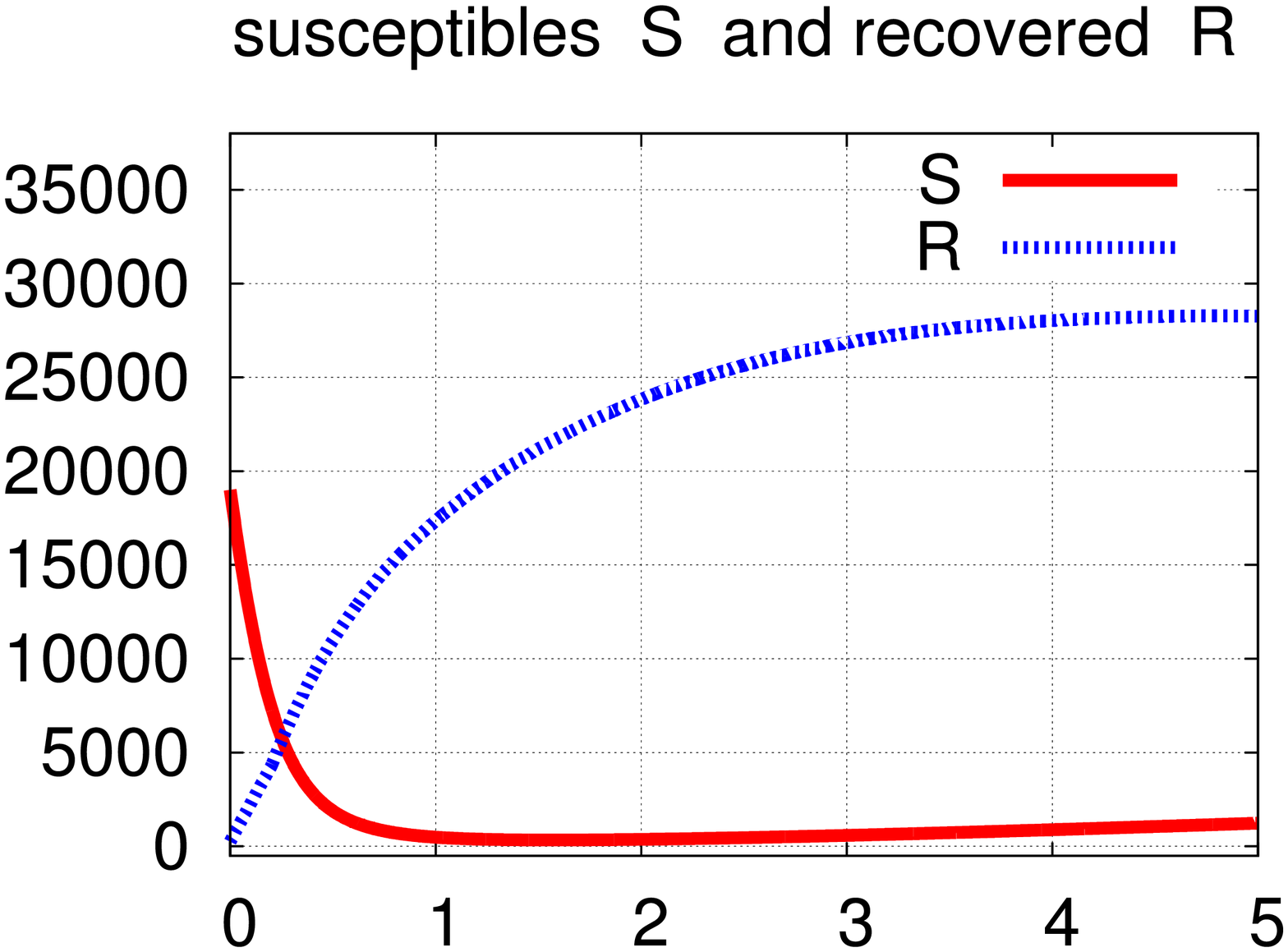}
\hspace{3mm}
\includegraphics[width=0.29\textwidth]{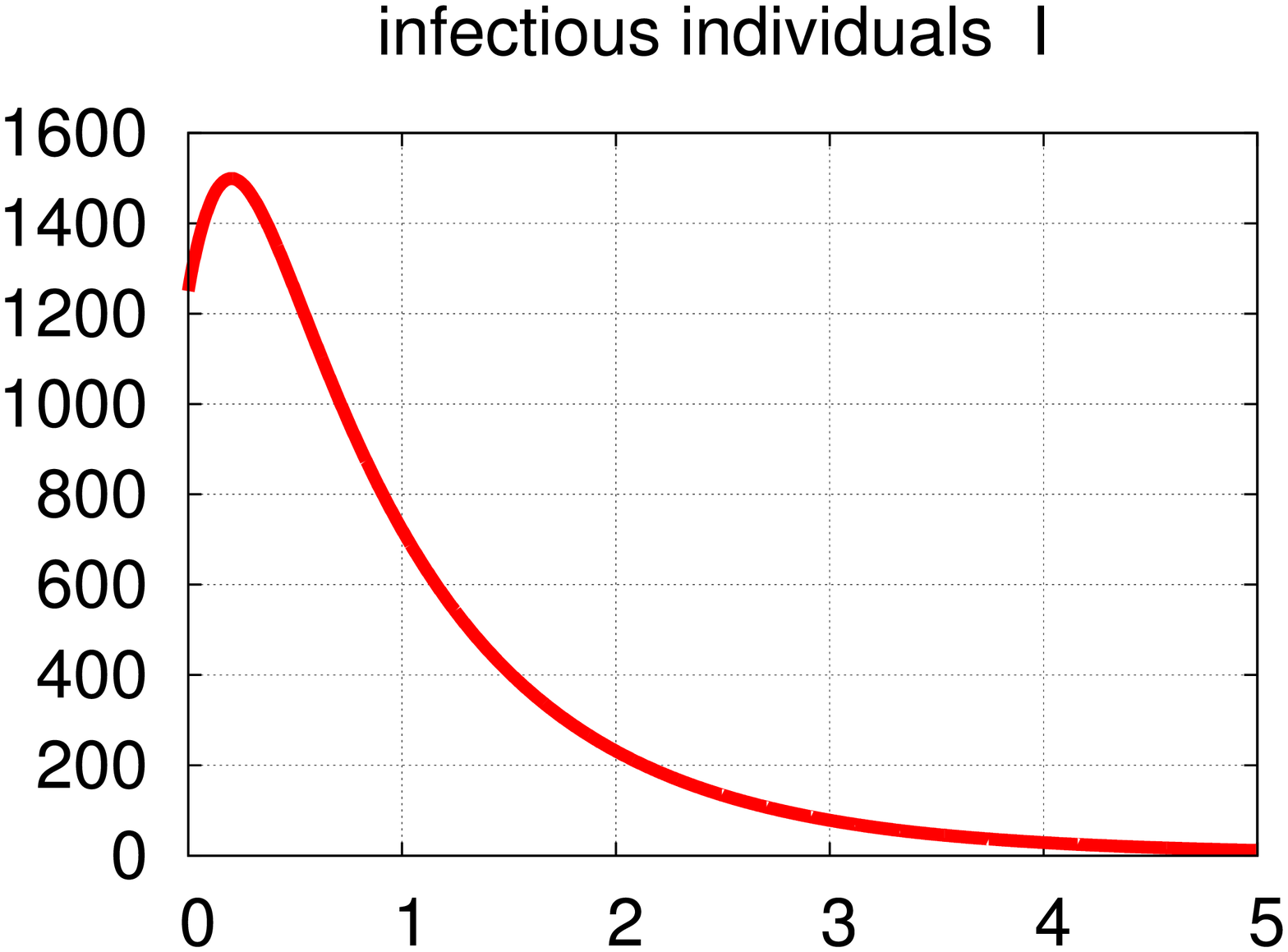}
\\[0.4cm]
\includegraphics[width=0.32\textwidth]{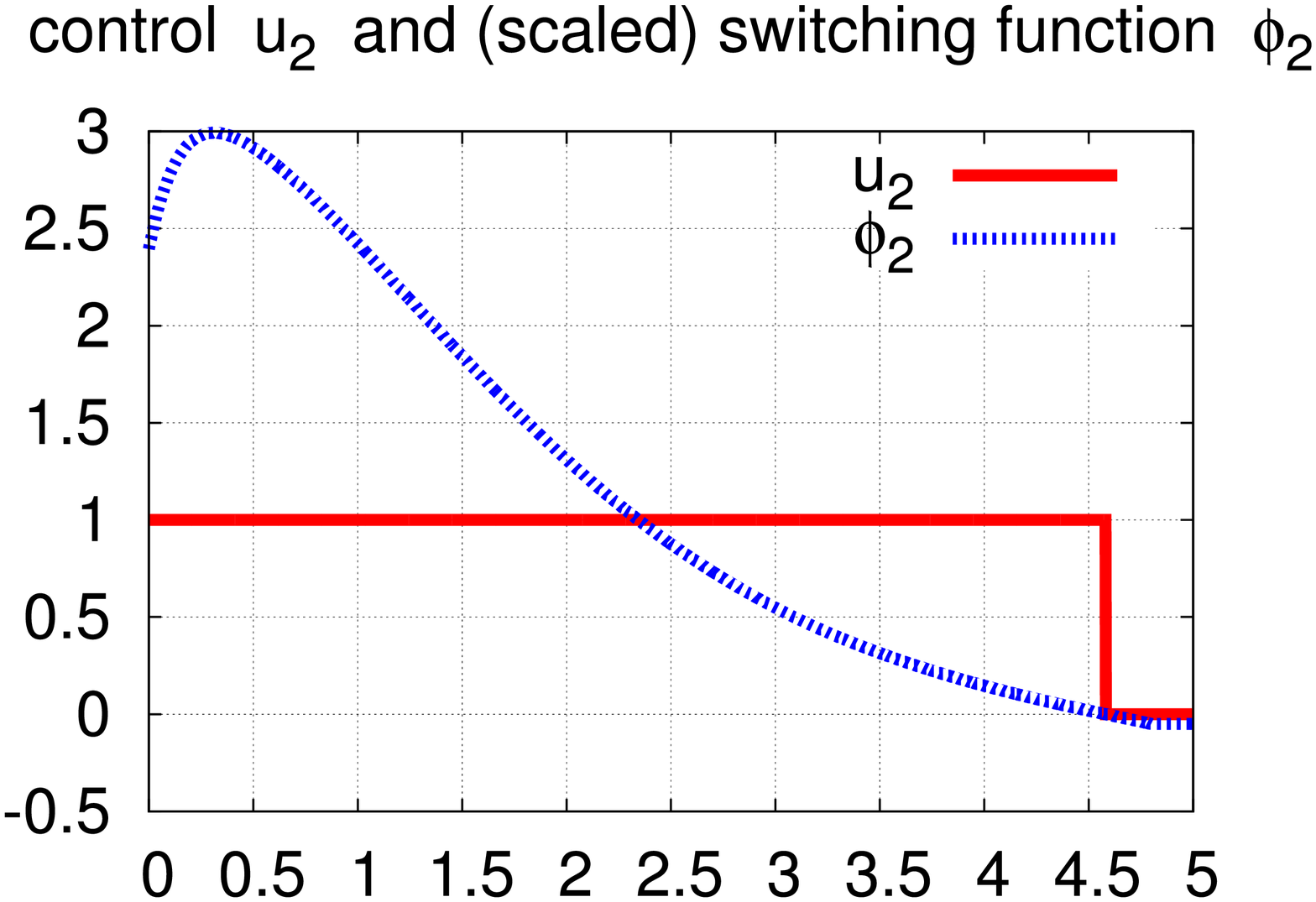}
\hspace{2mm}
\includegraphics[width=0.3\textwidth]{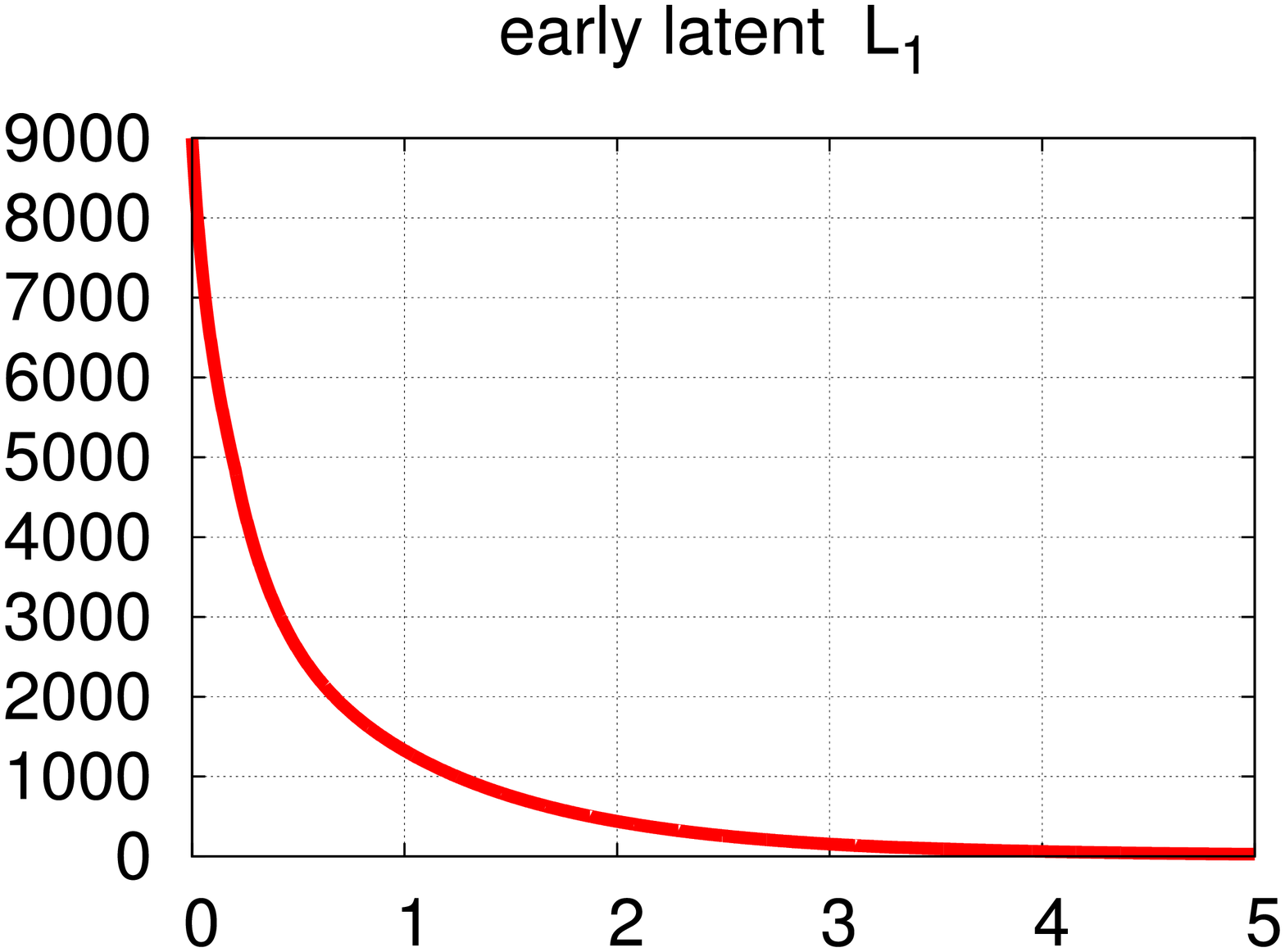}
\hspace{3mm}
\includegraphics[width=0.3\textwidth]{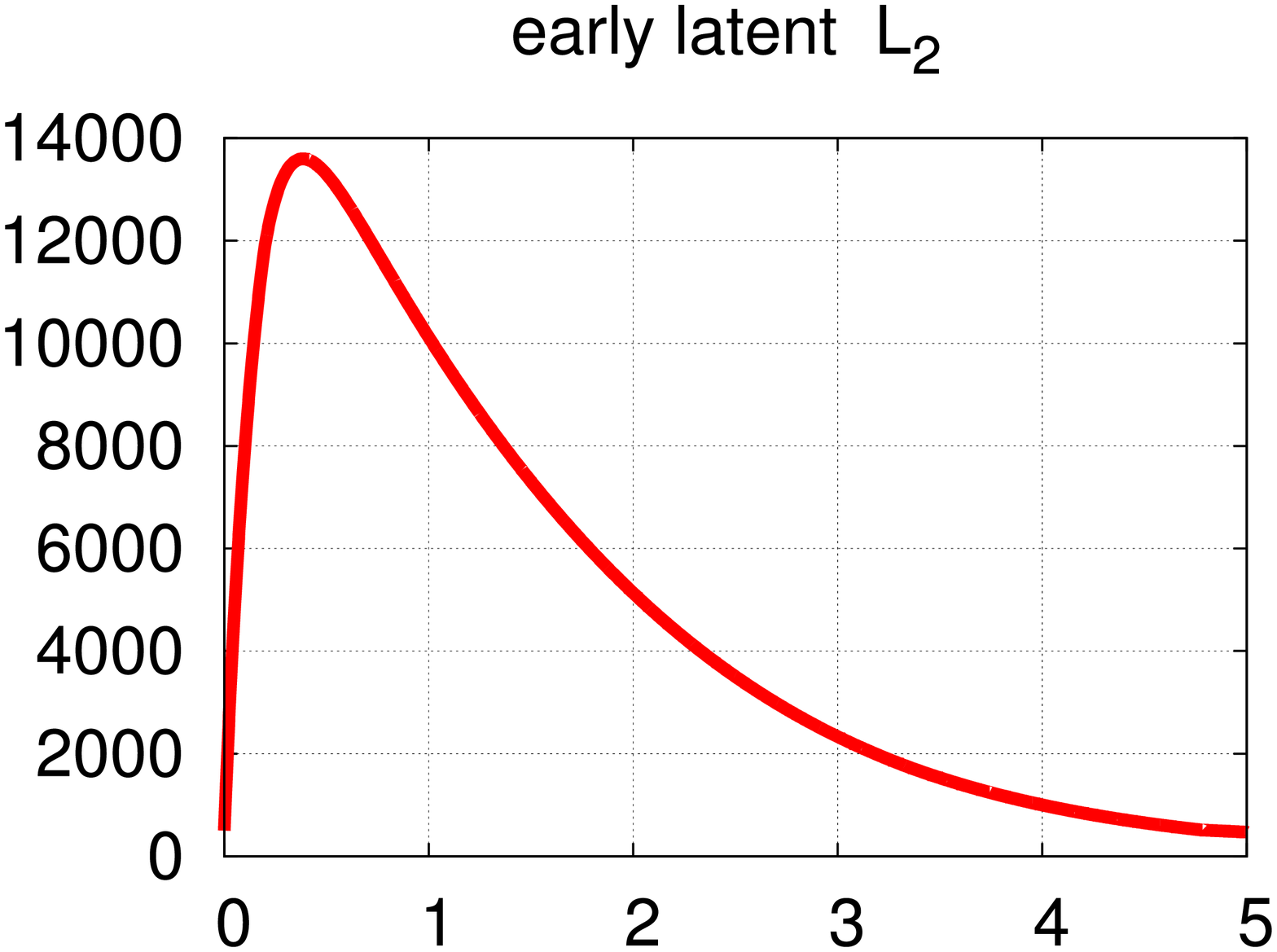}
\caption{Optimal control and state variables of the delayed TB model
with $L^1$-objective \eqref{J1-objective}, $W_1=W_2=50$ and delays
$d_I=0.1, d_{u_1}= d_{u_2}=0.2$. {\it Top row}: 
(a) control $u_1$ \eqref{eq:controlLawDelay}
and (scaled) switching function $\phi_1$ \eqref{eq:switch}
satisfying the control law \eqref{control-law} for $k=1$, 
(b) susceptible individuals $S$ and recovered individuals $R$, 
(c) infectious individuals $I$.
{\it Bottom row}: (a) control $u_2$ \eqref{eq:controlLawDelay}
and (scaled) switching function $\phi_2$ \eqref{eq:switch}
satisfying the control law \eqref{control-law} for $k=2$, 
(b) early latent $L_1$, (c) persistent latent $L_2$.}
\label{label:of:fig:5}
\end{figure}
As in the non-delayed problem, the switching functions satisfy the strict
bang-bang property related to the Maximum Principle:
$$
\phi_k(t) > 0 \quad \mbox{for} \;\; 0 \leq t < t_k \,,
\quad \dot\phi_k(t_k) <0 \,, \quad
\phi_k(t) < 0 \quad \mbox{for} \;\; t_k < t \leq T \;\; (k=1,2).
$$
However, we are not aware in the literature of any type of
{\it sufficient conditions} which could be applied to
the extremal solution shown in Figure~\ref{label:of:fig:5}.

We also compared the extremal solutions for the $L^1$-type objective
\eqref{J1-objective} and the $L^2$-type objective \eqref{J2-objective}.
Since the controls are very similar to those in Figure~\ref{label:of:fig:3},
they are not displayed here. Figure~\ref{label:of:fig:6} shows the extremal
controls for the $L^1$-objective \eqref{J1-objective} for the increased
weights $W_1=W_2=150$.
\begin{figure}[ht!]
\centering
\includegraphics[width=5.2cm,height=3.8cm]{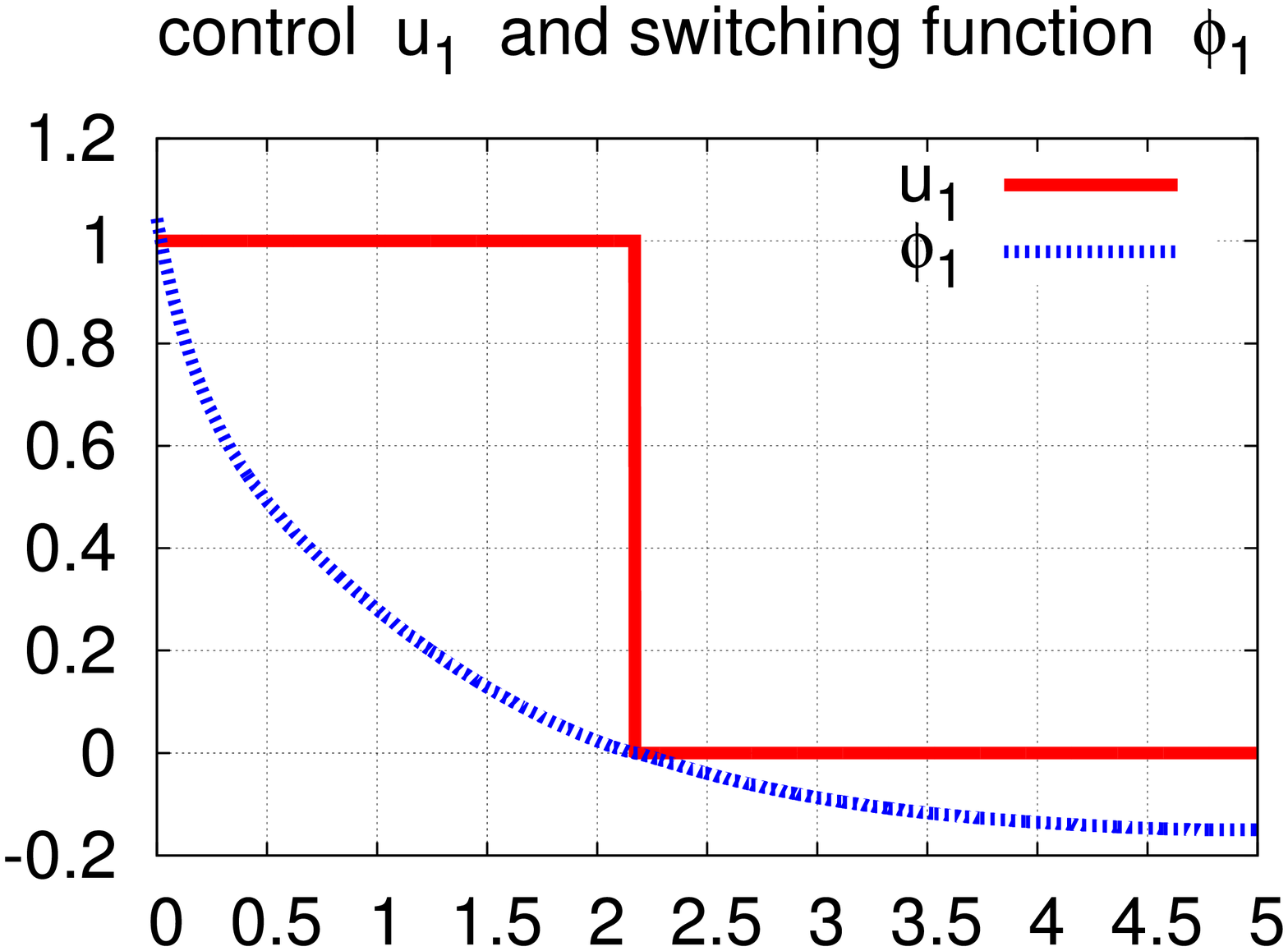}
\hspace{6mm}
\includegraphics[width=5.2cm,height=3.8cm]{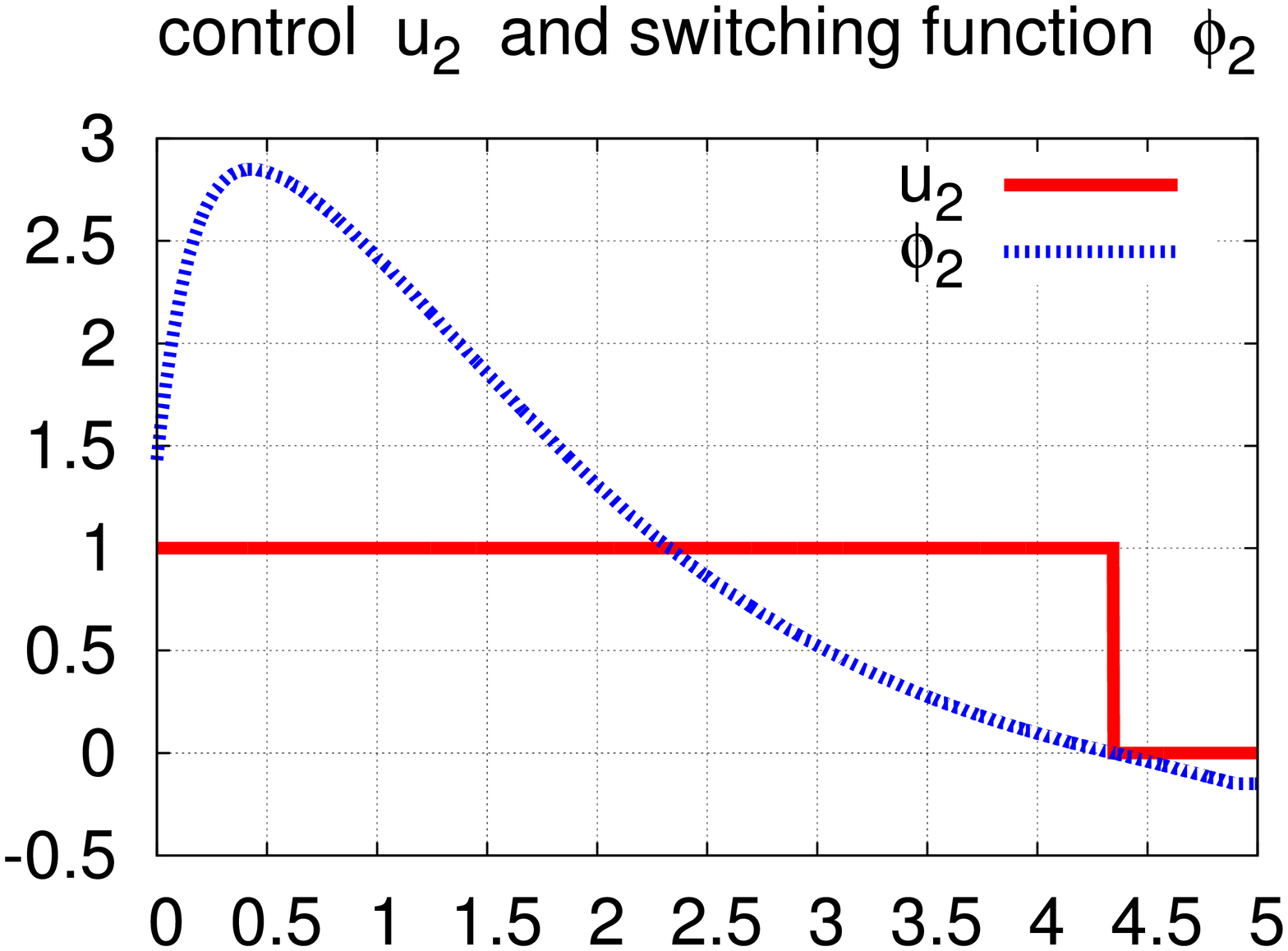}
\caption{Extremal controls for the delayed TB model with $L^1$
objective \eqref{J1-objective}, $W_1=W_2=150$ and delays
$d_I=0.1$, $d_{u_1}= d_{u_2}=0.2$. (a) control $u_1$ \eqref{eq:controlLawDelay}
and (scaled) switching function $\phi_1$ \eqref{eq:switch} satisfying 
the control law \eqref{control-law} for $k=1$,
(b) control $u_2$ \eqref{eq:controlLawDelay} 
and (scaled) switching function $\phi_2$ \eqref{eq:switch} 
satisfying the control law \eqref{control-law} for $k=2$.}
\label{label:of:fig:6}
\end{figure}

The most significant influence on the optimal controls
is exerted by the transmission coefficient $\beta$.
\begin{figure}[ht!]
\centering
\includegraphics[width=5.2cm,height=4.0cm]{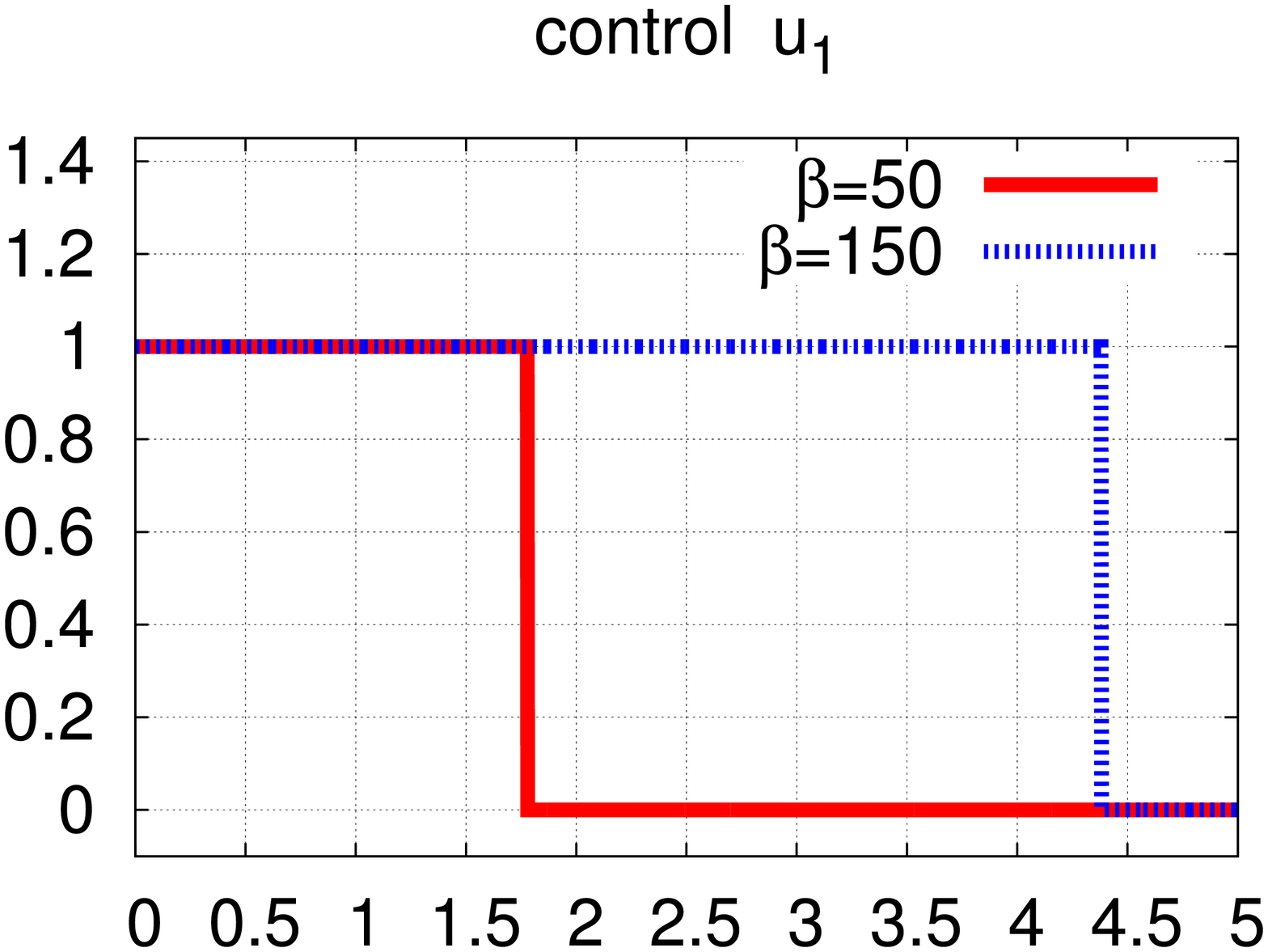}
\hspace{6mm}
\includegraphics[width=5.2cm,height=4.0cm]{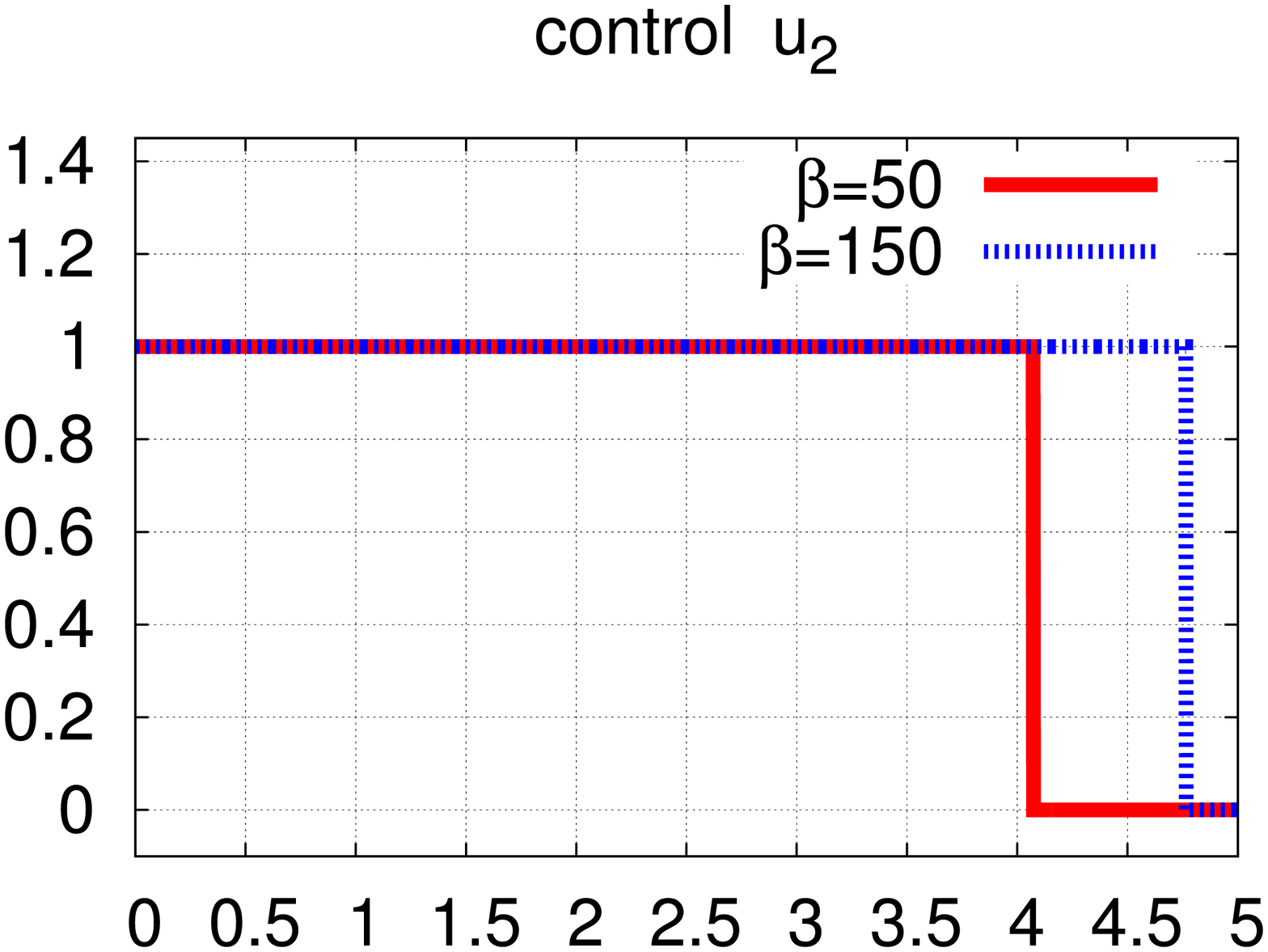}
\caption{Comparison of extremal controls for parameters $\beta=50$ and $\beta=150$
in the delayed TB model with $L^1$ objective \eqref{J1-objective}, weights
$W_1=W_2=150$ and delays  $d_I=0.1, d_{u_1}= d_{u_2}=0.2$.}
\label{label:of:fig:7}
\end{figure}
It can be clearly  seen that the increase of the transmission coefficient
$\beta$ triggers a substantial increase in the switching times $t_k$ of the
bang-bang controls $u_k$ for $k=1,2$ (cf. Figure~\ref{label:of:fig:7})
as could be expected from the equation for $\dot S$. Let us perform in this
case a more detailed sensitivity analysis of the trajectories with respect
to the parameter $\beta \in [50,150]$. We compute the extremal solutions
for a sufficiently fine grid of parameters $\beta \in [50,150]$ by a homotopy
method and plot the objective $J_1(x,u)$ and the terminal state
$(S(T),L_1(T),I(T),L_2(T),R(T))$ as a function of $\beta$.
The numerical results are shown in Figure~\ref{label:of:fig:8}.
\begin{figure}[ht!]
\centering
\includegraphics[width=0.31\textwidth]{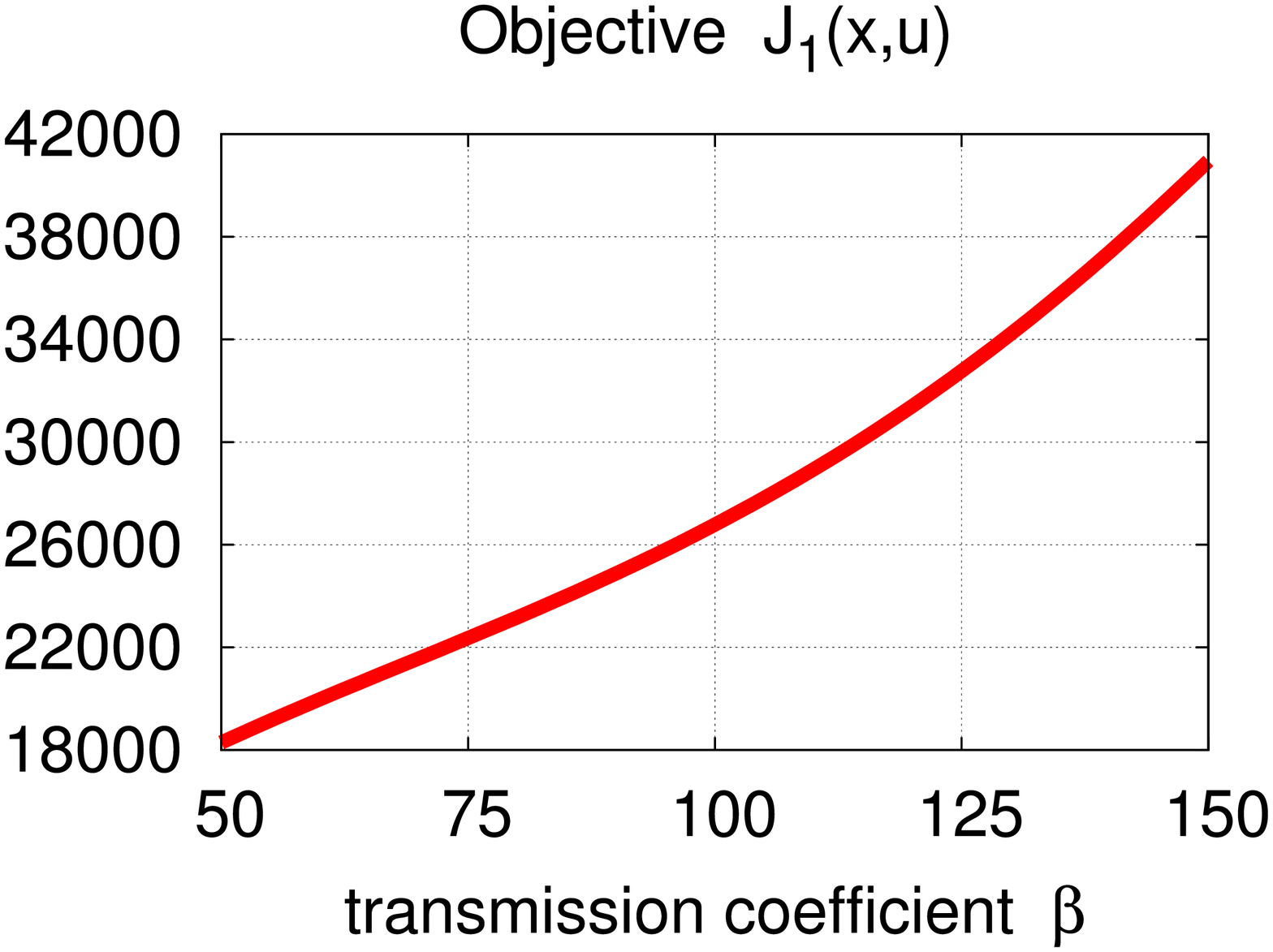}
\hspace{2mm}
\includegraphics[width=0.31\textwidth]{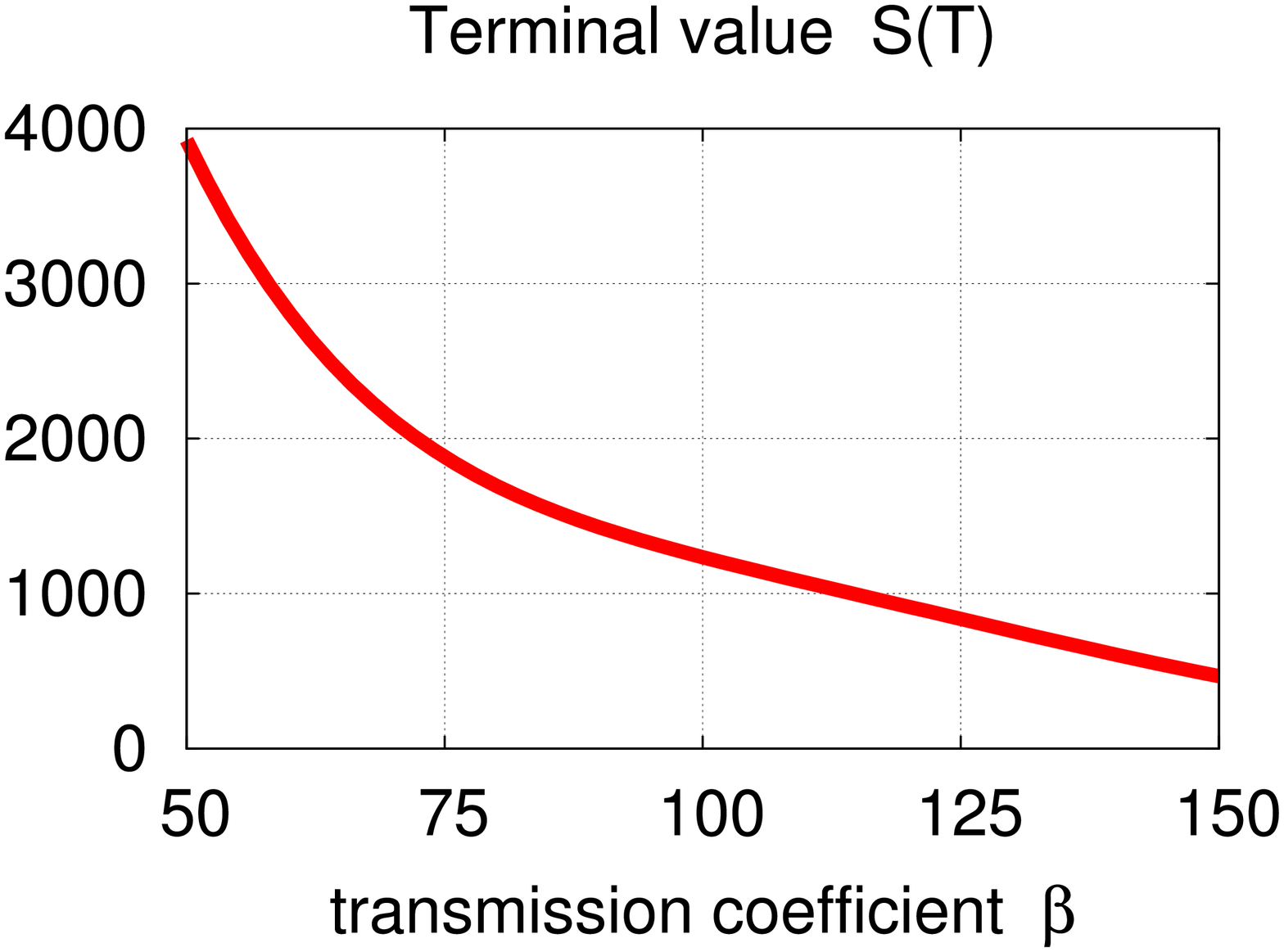}
\hspace{2mm}
\includegraphics[width=0.31\textwidth]{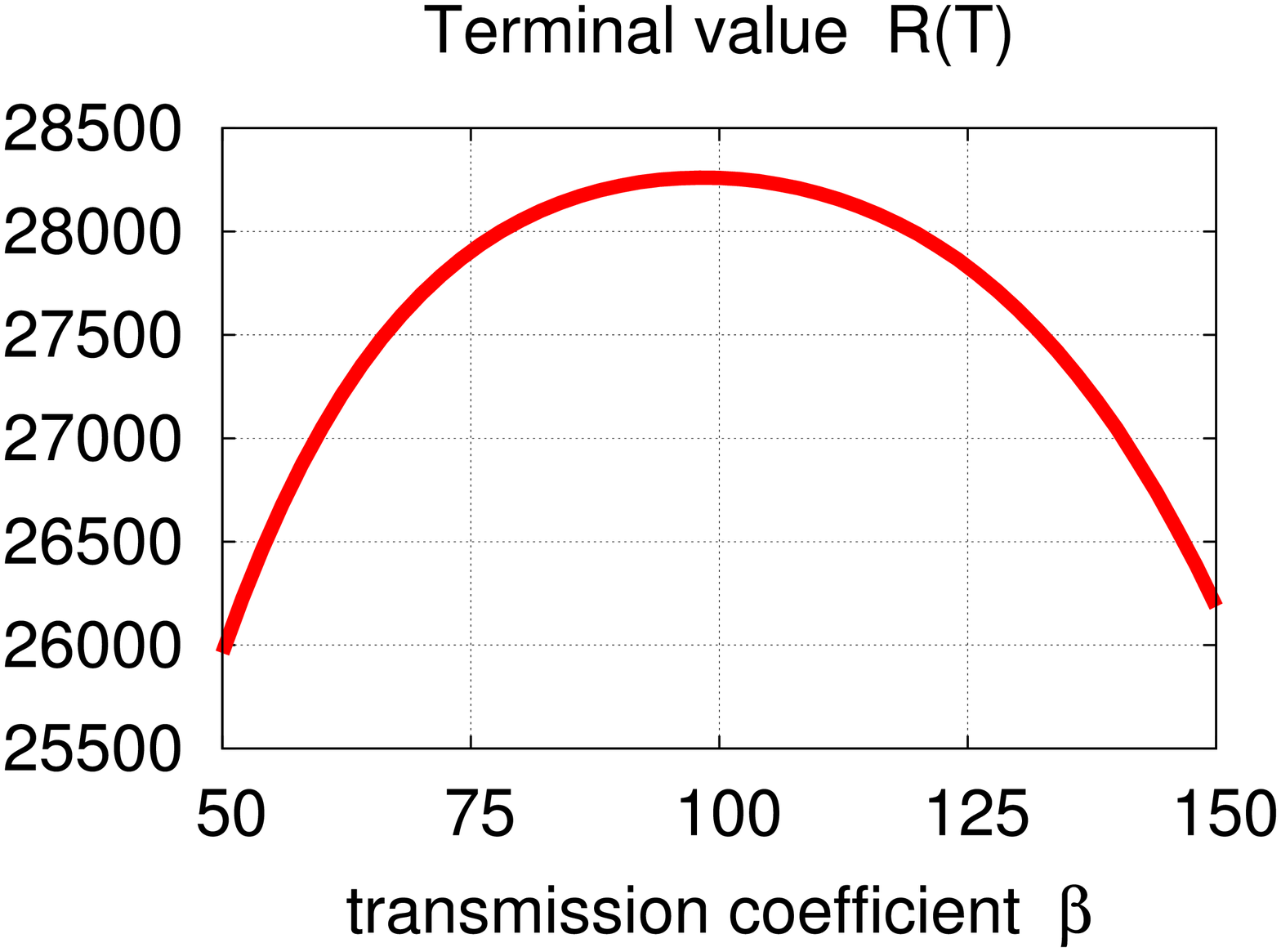}
\\[0.3cm]
\includegraphics[width=0.31\textwidth]{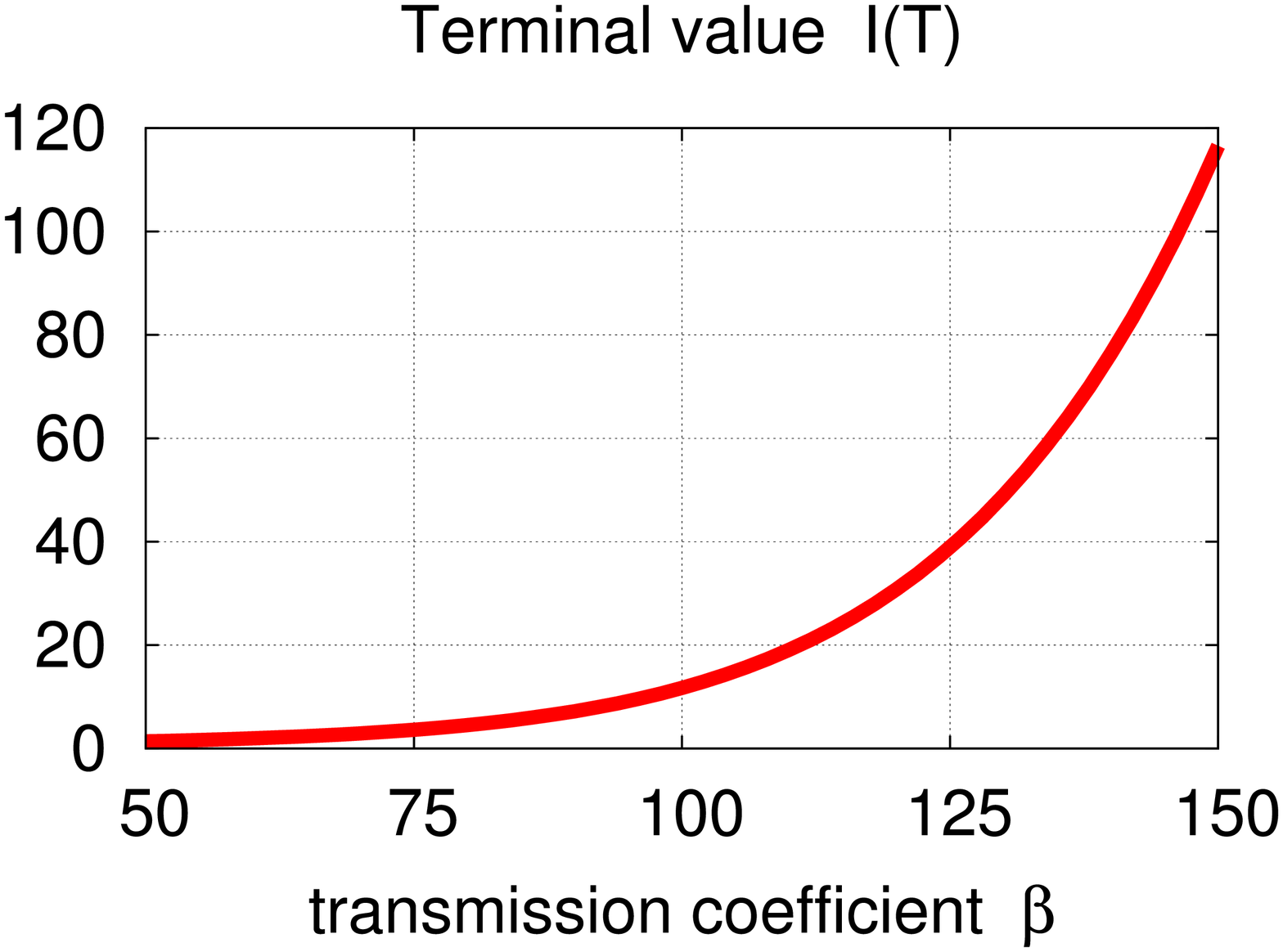}
\hspace{2mm}
\includegraphics[width=0.31\textwidth]{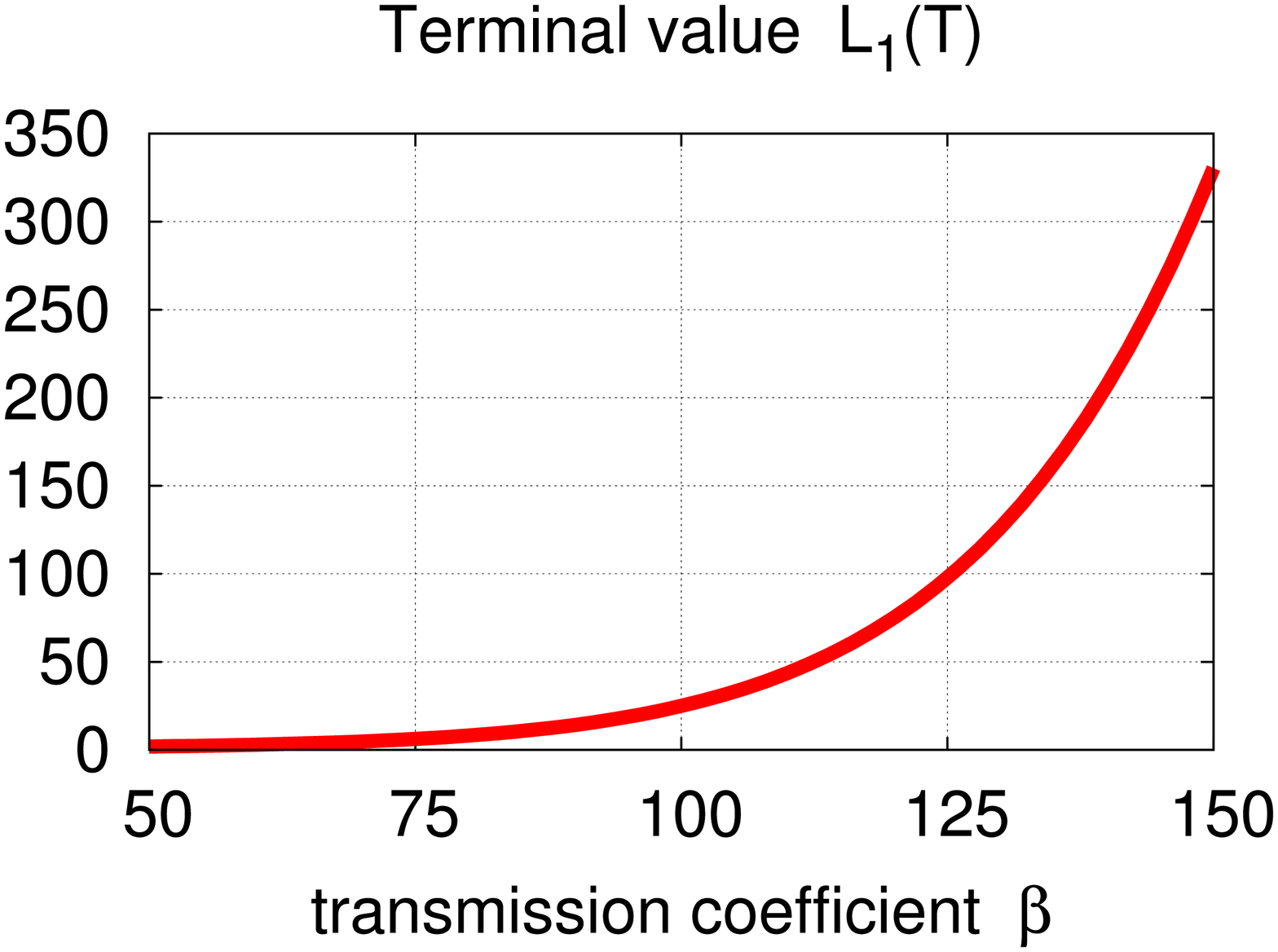}
\hspace{2mm}
\includegraphics[width=0.31\textwidth]{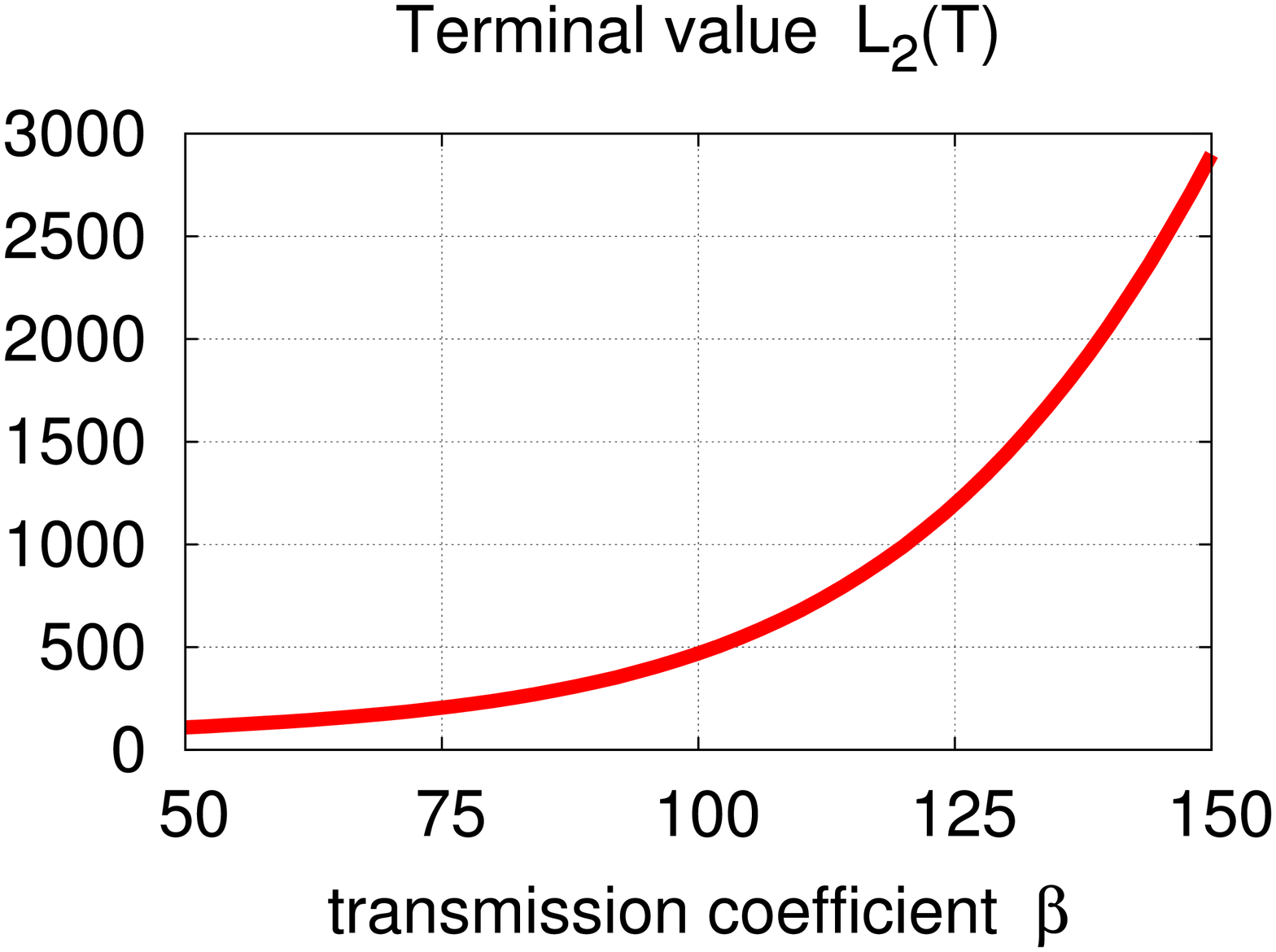}
\caption{Homotopic solutions of  the delayed TB model with $L^1$ objective
\eqref{J1-objective} and weights $W_1=W_2=50$ for parameters
$\beta \in [50,150]$. Displayed are the objective value $J_1(x,u)$
and the terminal states $S(T)$, $R(T)$, $I(T)$, $L_1(T)$, $L_2(T)$.}
\label{label:of:fig:8}
\end{figure}


\section{Conclusions}
\label{sec:7}

We introduced a discrete time delay $d_I$ in a mathematical model
for tuberculosis (TB), which represents the delay on the diagnosis
of active TB infection and commencement of treatment.
The delay on the diagnosis of active TB has important
negative consequences on TB control and eradication.
The later the treatment of active TB starts, the more people can be infected
and may die from TB. The introduction of a time delay on the state variable
of active TB infected individuals $I$ is therefore justified
from the epidemiological point of view.

When a time delay is introduced into a mathematical model, the stability
of its disease free and endemic equilibriums may change. We proved that
the disease free equilibrium (DFE) of the TB model with delay in the state variable
$I$ is unstable for any time delay $d_I \geq 0$, whenever the basic reproduction
number is greater than one. We derived  conditions under which the model
is locally asymptotically stable for $d_I = 0$ and proved that there exists
at least one positive time delay $d_I > 0$ such that the DFE is unstable
for $R_0 < 1$. Despite of this, we also proved that for the concrete
time delay $d_I = 0.1$, the set of parameters of Table~\ref{table:parameters},
and a transmission coefficient for which $R_0 < 1$, the DFE is locally
asymptotically stable. The value $d_I = 0.1$ (36.5 days) fits the data
reported in the literature for the delay in the diagnosis of active TB.
For the endemic equilibrium (EE), we considered that the transmission
coefficient $\beta$ is such that $R_0 > 1$ and proved the local stability
of the specific EE associated to our set of parameters
and discrete time delay $d_I = 0.1$.

We proposed an optimal control problem where the control system is the
mathematical model for TB with time delay in the state variable $I$ and where
the control functions $u_1$ and $u_2$ represent, respectively, the effort
on early detection and treatment of recently infected individuals $L_1$,
and the application of prophylactic treatment to persistent latent individuals
$L_2$. The objective is to minimize the number of individuals with active
and persistent latent TB as well as the cost associated to the implementation
of the control measures. Human immune system can take from two to eight weeks
to react to TB infection and detection. Moreover, prophylactic treatment of early
and persistent latent individuals may face resistance from patients and health
care services. Based on these facts, we introduced discrete time delays
in the controls $u_1$ and $u_2$. To our knowledge, this is the first time
an optimal control problem for TB with delays in the state and control variables
is investigated.

Firstly, we considered the non-delayed case ($d_I = d_{u1} = d_{u2}=0$)
and compared the solutions for $L^1$ and $L^2$ objectives. Our results
show that the optimal state variables are nearly identical for both objectives,
the control variables differing only on a terminal interval. The optimal values
of the objective functionals are also very close. We observed that for the $L^1$
objective the optimal control $u_2$ may have an additional zero arc at the
beginning if the weight constants associated to the implementation of the
control policies $u_1$ and $u_2$ are big enough. For the delayed optimal control
problem, our focus was on the $L^1$ objective since it seems to be the
more realistic and, comparing the extremal solutions of $L^1$ and $L^2$
objectives, the differences are similar to the non-delayed case. In the
delayed case, the switching times $t_1$ and $t_2$ are significantly smaller
and the costs are also smaller, when compared to the non-delayed case.
From an epidemiological point of view, when we introduce delays in the TB model,
the optimal solutions for the reduction of the number of individuals with active
and persistent latent TB infection are associated to control policies that are
less costly and can be implemented in a shorter period of time. Associated
to these positive facts, we observed that if a delay is introduced in the
state variable $I$ and in the controls, the number of individuals with active
TB infection at the end of five years is reduced in approximately 45.2 per cent,
when compared to the non-delayed case. Similarly, the number of persistent
latent individuals at the end of five years is also reduced in approximately
60.1 per cent. Moreover, the terminal number of susceptible and recovered
individuals is bigger in the delayed case. Through a sensitivity analysis,
we observed that the transmission coefficient $\beta$ has a significant influence
on the optimal controls and the cost functional, and that the number of active
infected individuals and the number of early and persistent latent individuals
also increase with $\beta$. The number of recovered individuals increases for
$\beta \in [50, 100]$ and decreases for $\beta \in [100, 150]$, which means
that for $R_0 > 3.1$ the control measures are no longer effective
and should be rethought.


\section*{Acknowledgments}

This work was partially supported by FCT within project TOCCATA,
reference PTDC/EEI-AUT/2933/2014. Silva and Torres were
also supported by CIDMA and project UID/MAT/04106/2013;
Silva by the post-doc grant SFRH/BPD/72061/2010.
The authors are very grateful to anonymous reviewers 
for their careful reading and helpful comments. 



\medskip

Received October 30, 2015; revised May 24, 2016; accepted June 26, 2016.

\medskip


\end{document}